\documentclass[pdflatex,sn-mathphys-num]{sn-jnl}

\usepackage{graphicx}%
\usepackage{multirow}%
\usepackage{amsmath,amssymb,amsfonts}%
\usepackage{amsthm}%
\usepackage{mathrsfs}%
\usepackage[title]{appendix}%
\usepackage{xcolor}%
\usepackage{textcomp}%
\usepackage{manyfoot}%
\usepackage{booktabs}%
\usepackage{algorithm}%
\usepackage{algorithmicx}%
\usepackage{algpseudocode}%
\usepackage{listings}%

\usepackage{tikz}
\usepackage{subcaption} 
\usepackage{amssymb,amsmath,fullpage,hyperref} 
\usepackage{amsthm}
\newtheorem{theorem}{Theorem}[section]
\newtheorem{corollary}{Corollary}[section]
\numberwithin{equation}{section}

\newtheorem{lem}{Lemma}[section]

\newtheorem{prop}{Proposition}[section]
\newtheorem{remk}{Remark}[section]
\newtheorem{definition}{Definition}[section]
\usepackage{pgfplots}

\usepackage{amsfonts,amsbsy,graphicx}
\usepackage{multirow}

\def\beq{\begin{equation}}
\def\eeq{\end{equation}}

\def\XXint#1#2#3{{\setbox0=\hbox{$#1{#2#3}{\int}$}
\vcenter{\hbox{$#2#3$}}\kern-.5\wd0}}

\def\strutdepth{\dp\strutbox}
\def \nw#1{\strut\vadjust{\kern-\strutdepth\vtop to0pt{\vss\hbox to\hsize {\hskip\hsize\hskip5pt$\leftarrow$\hss\strut}}}{\em \textcolor{blue}{#1}}}

\def\strutdepth{\dp\strutbox}
\def \AT#1{\strut\vadjust{\kern-\strutdepth\vtop to0pt{\vss\hbox to\hsize {\hskip\hsize\hskip5pt$\leftarrow$\hss\strut}}}{\em \textcolor{red}{#1}}}

\def\strutdepth{\dp\strutbox}
\def \BVR#1{\strut\vadjust{\kern-\strutdepth\vtop to0pt{\vss\hbox to\hsize {\hskip\hsize\hskip5pt$\leftarrow$\hss\strut}}}{\em \textcolor{orange}{#1}}}

\DeclareMathOperator\erfc{erfc}

\usepackage{amsmath}

\makeatletter
\newenvironment{varsubequations}[1]
 {%
  \addtocounter{equation}{-1}%
  \begin{subequations}
  \renewcommand{\theparentequation}{#1}%
  \def\@currentlabel{#1}%
 }
 {%
  \end{subequations}\ignorespacesafterend
 }
\makeatother

\usepackage{lmodern}
\pgfplotsset{compat=1.17}
\begin{document}


\title[The role of mutants in thespatio-temporal progression of inflammatory bowel disease: three classes of permanent form travelling waves]{The role of mutants in the spatio-temporal progression of inflammatory bowel disease: three classes of permanent form travelling waves}
\date{\today}
\author{\fnm{Blaine} \sur{Van Rensburg}}\email{\href{mailto:bxv114@student.bham.ac.uk}{bxv114@student.bham.ac.uk}}

\author{\fnm{David J.} \sur{Needham}}\email{\href{mailto:d.j.needham@bham.ac.uk}{d.j.needham@bham.ac.uk}}

\author{\fnm{Fabian} \sur{Spill}}\email{\href{mailto:f.spill@bham.ac.uk}
{f.spill@bham.ac.uk}}

\author{\fnm{Alexandra} \sur{Tzella}}\email{\href{mailto:a.tzella@bham.ac.uk}{a.tzella@bham.ac.uk}}

\affil{\orgdiv{School of Mathematics}, \orgname{University of Birmingham}, \orgaddress{ \city{Birmingham}, \postcode{B15 2TT}, \country{UK}}}

\abstract{Despite its high prevalence and impact on the lives of those affected, a complete understanding of the cause of inflammatory bowel disease (IBD) is lacking.
In this paper, we investigate a novel mechanism which proposes that mutant epithelial cells are significant to the progression of IBD since they promote inflammation and are resistant to death.
We develop a simple model encapsulating the propagation of mutant epithelial cells and immune cells which results from interactions with the intestinal barrier and 
bacteria. 
Motivated by the slow growth of mutant epithelial cells, and relatively slow response rate of the adaptive immune system, we combine   geometric singular perturbation theory with matched asymptotic expansions to determine the one-dimensional slow invariant manifold that characterises the leading order dynamics at all times beyond a passive initial adjustment phase. The dynamics on this manifold are controlled by a bifurcation parameter 
which depends on the ratio of {growth to decay rates} of all components except mutants and determines 
three distinct classes of  
  permanent-form travelling waves that describe the propagation of mutant epithelial and immune cells.  These are  obtained from scalar reaction-diffusion equations 
 with the reaction being  (i) a bistable nonlinearity  with a cut-off, (ii) a cubic Fisher nonlinearity and (iii)     a KPP or Fisher  nonlinearity.    
Our results suggest that 
mutant epithelial cells
are critical to the progression of IBD.  However, their effect on the speed of progression is subdominant.}

\keywords{mutant-mediated progression of disease, reaction–diffusion equations, geometric singular perturbation theory, matched asymptotics, travelling waves}



\maketitle

\section{Introduction}\label{sec_intro}
 Inflammatory bowel disease (IBD) is a chronic inflammatory condition
 which comes in two main forms termed Crohn's disease and the ulcerlative collitis. IBD has a high prevalence (of 0.3\% in Western countries \cite{NSHU})  and significantly impacts the quality of life of those affected \cite{knowles2018quality}. Moreover, the global incidence of IBD is rising, placing an increased burden on healthcare systems \cite{windsor2019evolving}. There is therefore a pressing need to better understand the mechanisms that drive  IBD and thus develop improved, more targeted treatments.
However, progress is hindered because complex 
spatio-temporal cues, such as those generated by 
chemical signalling molecules like cytokines \cite{altan2019cytokine}, 
control the response of immune cells and their interactions with 
  resident bacteria and the intestinal barrier. 
These interactions can be highly nonlinear and strongly influenced by genetic factors 
\cite{zhang2014inflammatory}, with relevant genes affecting the function of epithelial cells which constitute the central layer of the intestinal barrier \cite{mccole2014ibd}. Their combined effect on the generation and progression of IBD remains poorly understood.

 A growing number of mathematical and computational models have been proposed to elucidate the most relevant interactions of immune \cite{wendelsdorf2010model,lo2013mathematical,lo2016inflammatory} and epithelial cells \cite{wendelsdorf2010model} to the progression of IBD and make experimentally testable predictions to inform targeted treatments. 
However, these models do not include the effect of mutant epithelial cells.

 Recent experimental results suggest that these cells can play a significant role in the generation and progression of IBD \cite{OA}. Their presence can affect the normal function of the intestinal barrier,   permitting resident bacteria in the colon to cross and cause an immune cell response leading to inflammation. In the presence of inflammation, certain mutations have been shown to confer a selective advantage over wild-type epithelial cells, which potentially results in an initially isolated patch of mutant epithelial cells spreading along the intestinal barrier \cite{nanki2020somatic} which can lead to IBD.
The role of mutant epithelial cells on the generation and progression of IBD
cannot be captured by 
the current type of mathematical and computational models which assume that their components are spatially homogeneous.

\subsection{Model}
In this paper, we develop a simple system of reaction-diffusion equations to mechanistically model the {spatio-temporal dynamics} 
of the progression of mutant epithelial cells 
as they interact with the intestinal barrier, immune cells and bacteria. %
The model takes the form 
\begin{subequations}\label{eqn:MainSystem}
\begin{align}
\partial_{t}M&=D_{M}\partial_{xx}M+\gamma{}IM\left(K_{M}-M\right)\label{eqn:Mutant},\\
    \partial_{t}I&=D_{I}\partial_{xx}I+\alpha_{1}B\left(1-I/K_I\right)-\beta_{1}I\label{eqn:Immune},\\
    \partial_{t}\rho&=D_\rho\partial_{xx}\rho+\alpha_{2}I\left(1-\rho\right)-\beta_{2}\left(K_M-M\right)\rho\label{eqn:Barrier},\\
\partial_{t}B&=D_{B}\partial_{xx}B+r_{B}\rho-k_{B}{B}, 
\label{eqn:Bacteria}
\end{align}
\end{subequations}
where the dependent variables $M(x,t)$, $I(x,t)$, $B(x,t)$ are, respectively, the local spatial densities of mutant epithelial cells, immune cells, and bacteria at position $x$ at time $t$ with $(x,t)\in  \mathbb{R}\times\mathbb{R}^+$. 
The remaining dependent variable $\rho(x,t)$  
represents the permeability of the intestinal barrier; $\rho(x,t)=1$  corresponds to maximum  permeability, and $\rho(x,t)=0$ corresponding to an impermeable barrier.  
The first equation models mutant epithelial cells out-competing wild-type epithelial cells  in the presence of immune cells. We assume that the total density of epithelial cells (i.e the sum of mutant and wild-type cell densities) is constant and the overall rate of this interaction is provided by the law of mass action with reaction rate $\gamma$. This the reaction term is of the same form as the standard Fisher-KPP reaction, where the selective advantage of mutant epithelial cells, represented by $\gamma{}I$, is due to inflammation.
The second equation models the change in the immune cell density. We assume 
that
immune cells are recruited by bacteria at rate $\alpha_1$ and otherwise, decay linearly at rate $\beta_1$ due to death or migration away from the intestine. 
  A carrying capacity $K_I$  accounts for saturation in recruitment.
The third equation assumes that the permeability of the intestinal barrier increases proportionally to the density of immune cells since excess inflammation is known to be damaging \cite{thoo2019keep}.
At the same time, the permeability of the intestinal barrier decreases in the presence of wild-type epithelial cells which secrete antibodies that suppress the transfer of bacteria across the intestinal barrier \cite{johansen1999absence,mantis2011secretory}. The corresponding growth and decay rates are, respectively,  $\alpha_2$ and $\beta_2$.
Finally, 
the  source for bacteria 
is proportional to  the permeability of the intestinal barrier at rate $r_B$ and we assume that bacteria  decay linearly at rate  $k_B$ due to clearance by other immune cells, such as macrophages, which we do not explicitly model \cite{bain2014macrophages}. 
We assume that mutant epithelial cells, immune cells, and bacteria diffuse with diffusivities given by $D_M$, $D_I$, $D_B$ respectively.  For simplicity, we also assume  
 that the permeability of the intestinal barrier undergoes diffusion.
As we see in the theory that unfolds in the paper this diffusivity parameter can in effect be set to zero 
but is here included for generality.

  As initial conditions we take $M(x,0)=M_0 H(-x)$
  and 
  $I(x,0)=I_0 H(-x)$ where $H$ corresponds  to the Heaviside function
  and 
  assume that  the permeability of the intestinal barrier and the density of bacteria initially satisfy 
 $(\rho(x,0),B(x,0))\in(0,1)\times(0,r_B/k_B)$.
 Thus we assume that initially there are no mutant epithelial cells or immune cells beyond $x=0$.

\subsection{Summary of results and outline}
The slow growth of  mutant epithelial cells \cite{olpe2021diffusion}  and relatively slow response of the adaptive immune system \cite{janeway2001immunobiology} motivates the use of geometric singular perturbation theory followed by matched asymptotic expansions which we use to deduce that beyond a short and passive induction period, the spatio-temporal dynamics of mutant epithelial and immune cells  are attracted by a one-dimensional  slow invariant manifold. We show that on this invariant manifold  the reduced evolutionary system supports the development of travelling waves provided that the initial density of mutants $M_0$ exceeds a threshold value determined as $\text{max}(K_M(1-\sigma),0)$ . This value decreases with the key parameter $\sigma$, where  
\begin{equation}\label{eqn:SigmaExpression}
\sigma=\frac{\alpha_1\,\alpha_2\,r_B}{\beta_1^{}\,\beta_2^{}k_{B}^{}\,K_M^{}}
\end{equation} represents the ratio of growth to decay rates of the system's components except mutant epithelial cells, and is scaled by the constant total epithelial cell density.

The {reduced} dynamics are governed by 
scalar reaction-diffusion equations with the reaction being (i) $0<\sigma<1$: a bistable nonlinearity that is effectively deactivated below a threshold value (see \cite{dumortier2010geometric,popovic2012geometric} for similar cut--off reactions), (ii) $\sigma=1$: a 
cubic Fisher nonlinearity  \cite{billingham1991note}, or (iii) $\sigma>1$: a classical Kolmogorov–Petrovskii–Piskunov (KPP) or Fisher 
nonlinearity \cite{fisher1937wave,kolmogorov1937etude}. 
These {parabolic evolution} equations support {(a possibly one-parameter family of)} non-negative permanent form travelling waves (PTWs) moving in the positive $x$ direction whose speed {(unique or minimum available)} $v$ is determined from a nonlinear boundary value problem. In case (iii) there are two boundary value problems where the first is independent of, and the second is dependent on, mutant epithelial cells, and each supports its own family of PTWs. The distinct difference between the three cases is that in case (i) the resulting propagation speed is unique, whereas in case (ii) and (iii) there is a continuum of speeds $[v_m,\infty)$, where $v_m$ denotes the minimum propagation speed. Furthermore, in case (iii) the ratio of the propagation speeds of the mutant epithelial cell independent PTW and the mutant epithelial cell dependent PTW is proportional to $1/\sqrt{\delta}$ where $\delta$ is a principal small parameter introduced in the next section. 
In each of cases (i), (ii), and (iii) the propagation speed is expressed  either explicitly or approximated in asymptotic limits, using existing expressions    \cite{billingham1991note,volpert1994traveling, LNBK}  or expressions  developed herein.

Our results suggest that mutant epithelial cells can significantly affect the progression of IBD, in the sense that the initial density of mutant epithelial cells $M_0$  must exceed the threshold value max$(K_M(1-\sigma),0)$ for the propagation of inflammation and progression of IBD to occur. This threshold is trivially exceeded when the ratio of growth rates to decay rates exceeds unity ($\sigma\geq{}1$). Therefore, our model suggests {that $\sigma$ is a key dimensionless parameter controlling disease propagation, and in particular,} for therapies to be successful they must reduce the value of $\sigma$. This can be achieved by targeting processes on which the parameters in the expression of $\sigma$ depend, such as immune cell recruitment, immune cell clearance, or that reduce damage done to the epithelial barrier.

The paper is organized as follows. 
In Section \ref{sec_model_reduction} we non-dimensionalise the model and identify two small dimensionless parameters that allow us to reduce its spatio-temporal dynamics to a two-dimensional slow invariant manifold {embedded in the four-dimensional phase space of dependent variables.} 
In Section 
\ref{sec_IVP} we  study  the initial value problem for the reduced two-dimensional model with front-like initial data
and determine necessary initial conditions for the activation of permanent form travelling wave solutions. In Section \ref{sec_PTW} we apply matched asymptotic expansions to examine in detail the structure of the permanent form travelling wave solutions which may emerge at long times. Technical details concerning the evolution of the initial value problem, and the preliminary results concerning the basic properties of the PTW solutions, are relegated to Appendix \ref{app_A} and Appendix \ref{app_B} respectively.
We ultimately reduce the dynamics to a one-dimensional slow invariant manifold, whose leading order structure we determine, and obtain asymptotic expressions for the propagation speeds, with explicit asymptotic expressions for limiting values of $\sigma$ provided in  Appendix \ref{app_C} and using existing expressions    \cite{billingham1991note,TNT,volpert1994traveling, LNBK}.  {In Section \ref{sec_numerics} we present numerical approximations obtained for each type of travelling wave that verify the theory developed in the preceding sections}. Finally,  we discuss and interpret our results in the concluding  Section
\ref{sec_conclusions}.

\section{Model Reduction}\label{sec_model_reduction}

In this section we first non-dimensionalise   model (\ref{eqn:MainSystem}) using suitable natural scales. This enables us to identify key dimensionless parameters in the model, and to estimate their magnitudes. To this end we introduce dimensionless variables as
\begin{align*}
&M^{*}=\frac{M}{K_{M}},~~I^{*}=\frac{I}{K_{I}},~~B^{*}=\frac{k_B}{r_B}B,~~
    x^{*}=\sqrt{\frac{\gamma{}K_MK_I}{D_M}}x,~~t^{*}=\gamma{}K_IK_M{t}.
\end{align*}
{This gives rise to six dimensionless parameters appearing in the dimensionless form of the model (\ref{eqn:MainSystem}), namely,
\begin{align*}
    {D}^{*}_{i}&=\frac{D_{i}}{D_{M}},~~   \beta_1^{*}=\frac{K_Ik_B}{\alpha_1r_B}\beta_1,~~\alpha^{*}_{2}=\frac{K_{I}}{k_B}\alpha_{2},~~ \beta^{*}_{2}=\frac{K_{M}}{k_B}\beta_{2},\\
    {\varepsilon}&=\frac{K_{M}K_{I}\gamma}{k_B},~~\delta=\frac{K_I}{\alpha_1{}r_B}\varepsilon.
\end{align*}
In what follows we omit the asterisks for notational convenience. The slow growth of mutant epithelial cells and the slow response time of the adaptive immune system relative to the time scales of the other processes determine first that $\epsilon$ is very small. In addition, in general we expect that $K_I(\alpha_1r_B)^{-1}$ will be large, but not too large, so that, naturally, we have $0<\epsilon \ll \delta \ll 1$. All other dimensionless parameters will be taken as being of $O(1)$ in magnitude.}
In dimensionless form, the final model equations can be written as the four dimensional reaction-diffusion system
\begin{varsubequations}{RDS}\label{eqn:RDS}
\begin{equation}
    \mathbf{U}_t = \mathbf{D}\mathbf{U}_{xx} + \mathbf{F}(\mathbf{U}),~~(x,t)\in D_{\infty}\equiv \mathbb{R}\times\mathbb{R}^+,\label{eqn:RDSa}
\end{equation}
with $x$ being the scalar spatial variable and $t$ being time. The dependent variables are $\mathbf{U}\equiv (M,I,\rho,B) \in \mathbb{R}^4$ whilst $\mathbf{F} \equiv (F_1,F_2,F_3,F_4)\in C^1(\mathbb{R}^4,\mathbb{R}^4)$ is the prescribed reaction function and $\mathbf{D}$ is the constant, diagonal and positive definite $4\times4$ diffusion matrix. We have,
    \begin{equation}
    \mathbf{D} = \mathbf{diag}[1,D,D_{\rho},D_B],\label{eqn:RDSb}
\end{equation}
with each entry strictly positive, whilst the reaction function has,
\begin{equation}
    F_1(M,I,\rho,B) = IM(1-M),\label{eqn:RDSc}
\end{equation}
\begin{equation}
    F_2(M,I,\rho,B) = \delta^{-1}(B - (B+\beta_1)I),\label{eqn:RDSd}
\end{equation}     
 \begin{equation}
    F_3(M,I,\rho,B) = \epsilon^{-1}(\alpha_2I - (\alpha_2I + \beta_2(1-M))\rho),\label{eqn:RDSe}
\end{equation}
\begin{equation}
    F_4(M,I,\rho,B) = \epsilon^{-1}( \rho - B),\label{eqn:RDSf}
\end{equation}
\end{varsubequations}
with $(M,I,\rho,B) \in \mathbb{R}^4$. We will henceforth refer to this strictly parabolic, semilinear evolutionary system as \eqref{eqn:RDS}. The model parameters $\beta_i (i=1,2)$, $\alpha_2$, $\delta$ and $\epsilon$ are all positive. A $\emph{global solution}$ to \eqref{eqn:RDS} is a function $\mathbf{U}:\overline{D}_{\infty} \to \mathbb{R}^4$ which has, $\emph{for each}$ $T>0$, $\mathbf{U} \in PC(\overline{D}_{T}) \cap C^{1,2}(D_T) \cap L^{\infty}(\overline{D}_T)$, with $D_T = \mathbb{R}\times (0,T]$. We now consider some preliminary results for \eqref{eqn:RDS}. First, it is straightforward to check that $\overline{Q}_R \subset \mathbb{R}^4$, where 
\begin{equation*}
    Q_R = \{(M,I,\rho,B)\in \mathbb{R}^4: (M,I,\rho,B)\in (0,1)^4\}, \label{eqn1.7}
\end{equation*}
is a positively invariant rectangle for \eqref{eqn:RDS}. Thus, with initial data 
\begin{equation*}
    \mathbf{U}(x,0) = \mathbf{U}_0(x),~~x\in \mathbb{R}, \label{eqn1.8} 
\end{equation*}
where $\mathbf{U}_0\in PC(\mathbb{R})\cap L^{\infty}(\mathbb{R})$, and 
\begin{equation}
    \mathbf{U}_0(x)\equiv (M_0(x),I_0(x),\rho_0(x),B_0(x))\in \overline{Q}_R~~\forall~~x\in \mathbb{R}, \label{eqn1.9}\tag{C}
\end{equation}
we can conclude, for any $T>0$, that any solution to \eqref{eqn:RDS} with this initial data must have,
\begin{equation}
    \mathbf{U}(x,t) \in \overline{Q}_R~~\forall~~(x,t)\in \overline{D}_T.  \label{eqn1.10}
\end{equation}

We will restrict attention throughout to initial data which satisfy condition (C).
In this situation, an immediate consequence of the $\emph{a priori}$ bounds provided in (\ref{eqn1.10}) is that the initial value problem for \eqref{eqn:RDS} has a global solution, on $\overline{D}_{\infty}$, which is unique, and  satisfies (\ref{eqn1.10}) (this follows from the classical uniqueness and global existence theory for regular, strictly parabolic, semilinear systems, see, for example,  Smoller \cite{Smoll}, chapter 14). 

We further restrict attention to the natural situation when $0<\epsilon\ll\delta\ll1$, with the object of determining qualitative and quantitative information concerning this global solution to \eqref{eqn:RDS}. We can formalise this by considering $\epsilon = o(\delta)$ in the limit $\delta\to 0$ in \eqref{eqn:RDS}, with all other parameters being formally of $O(1)$. In this limit we return to \eqref{eqn:RDS}, and take $\delta$ small, whilst allowing $\epsilon$ to become very much smaller than $\delta$, and denote $\mathbf{U}:\overline{D}_{\infty}\to \overline{Q}_R$ as the solution to the initial value problem. In this limit there is an asymptotic separation of time scales between the first two components (the slow components) and the last two components (the fast components) of \eqref{eqn:RDS} which can be readily exploited via the geometric singular perturbation theory. We first observe that,
\begin{equation*}
    (F_3)_{\rho}(M,I,\rho,B) = - \epsilon^{-1} (\alpha_2I + \beta_2(1-M)) < 0,  \quad
    (F_3)_B(M,I,\rho,B) = 0,
\end{equation*}
and
\begin{equation*}
    (F_4)_{\rho}(M,I,\rho,B) = \epsilon^{-1},
\quad
   (F_4)_B(M,I,\rho,B) = - \epsilon^{-1} <0, \label{eqn1.12} 
\end{equation*}
for all $(M,I,\rho,B)\in Q_R$. It then follows that, for $0<\epsilon\ll\delta\ll1$, \eqref{eqn:RDS} has a smooth two-dimensional, global and temporally stable, slow invariant
manifold in the invariant rectangle $\overline{Q}_R$ of the $(M,I,\rho,B)$-phase space, onto which the dymanics of \eqref{eqn:RDS} is (exponentially rapidly)  attracted on the fast time scale $t=O(\epsilon)$, and is thereafter constrained to remain on the slow time scale $t\ge O(1)$ (see, for example,  Jones \cite{arnold1995geometric}). A trivial direct calculation determines the slow invariant manifold as $\epsilon \to 0$, in the form,
\begin{equation}
    \mathcal{M}_i = \{(M,I,\rho,B)\in \overline{Q}_R:(M,I)\in [0,1]^2, \rho=B=B_i(M,I)\}, \label{eqn1.13}
\end{equation}
with,
\begin{equation*}
    B_i(M,I) = \frac{\alpha_2I}{(\alpha_2I + \beta_2(1-M))}.  \label{eqn1.14}
\end{equation*}
We see that the slow invariant manifold $\mathcal{M}_i$ is a $\emph{graph}$ in $\overline{Q}_R$, with independent coordinates $(M,I)\in [0,1]^2$. It should also be noted that the approximation to $B_i(M,I)$ fails in an $O(\epsilon)$ neighbourhood of the point $(M,I) = (1,0)$, and this can be regularised by the inclusion of a local inner region in the graph; however, we will see that this point and its local neighbourhood are not involved in the ensuing dynamics, and so we need not consider this further. 

We can now reduce the dimension of \eqref{eqn:RDS} to the associated two-dimensional system on the slow invariant manifold. On substituting for $A$ and $\rho$ from their graphs in (\ref{eqn1.13}) into the first two components of \eqref{eqn:RDS}, we obtain the two-dimensional reduced system as, in component form,
\begin{varsubequations}{LDS}\label{eqn:LDS}
   \begin{equation}
    M_t = M_{xx} + IM(1-M), \label{eqn1.15}\tag{LDSa}
\end{equation}
\begin{equation}
    I_t = DI_{xx} + \delta^{-1}f(M,I), \label{eqn1.16}\tag{LDSb}
\end{equation} 
\end{varsubequations}
for $(x,t)\in D_{\infty}$ with $(M,I)\in [0,1]^2$. Here
\begin{equation}
   f(M,I) = \frac{a(\sigma)I(b(\sigma)(\sigma-1) +b(\sigma)M - I)}{(\alpha_2I + \beta_2(1-M))}    \label{eqn1.17} 
\end{equation}
for $(M,I)\in [0,1]^2$.
It has been instructive to introduce the parameters,
\begin{equation}
    \sigma = \frac{\alpha_2}{\beta_1 \beta_2}, \quad 
    a(\sigma) = \beta_1(\beta_2\sigma + \alpha_2),
\quad
    b(\sigma) = \frac{\beta_2}{(\beta_2\sigma + \alpha_2)}.
\end{equation}
The initial conditions associated with \eqref{eqn:LDS} are, from (\ref{eqn1.9}),
\begin{equation*}
    (M(x,0),I(x,0)) = (M_0(x),I_0(x)),~~x\in \mathbb{R}. \label{eqn1.23}
\end{equation*}
As should be anticipated, $[0,1]^2$ is an invariant rectangle for \eqref{eqn:LDS} on the $(M,I)$-phase plane, whilst it may be further verified  directly that, in fact, for a given small $\delta$, the closed, convex clipped rectangle $R(\delta) \equiv [0,1]^2 \setminus N(\delta)$, with $N(\delta) = \{(M,I): (1-\delta)<M\le1 ~~\text{with}~~ 0\le I < M -(1-\delta)\}$ is also an invariant region for \eqref{eqn:LDS}, \emph{provided} $\sigma \ge \delta$, which we will take to be the case throughout the rest of the paper.. Without loss of generality, we will now work with \eqref{eqn:LDS} on $R(\delta)$, for which  the positive invariance of $R(\delta)$ guarantees that the initial value problem for \eqref{eqn:LDS} has a unique, global solution on $\overline{D}_{\infty}$, say $(M,I):\overline{D}_{\infty} \to R(\delta)$, and, by construction, this solution will approximate the solution to the corresponding initial value problem for \eqref{eqn:RDS}, when $\epsilon$ is small, in the following way: on the slow invariant manifold,
\begin{equation*}
    \left\Vert(M(\cdot,t),I(\cdot,t))|_{(RDS)} - (M(\cdot,t),I(\cdot,t))|_{(LDS)}\right\Vert_{\infty} = O(\epsilon)~~\text{as}~~\epsilon \to 0, \label{eqn1.24}
\end{equation*}
uniformly for $t\ge O(1)$. Here the subscripts \eqref{eqn:RDS} and \eqref{eqn:LDS} refer to solutions of the corresponding initial value problems for evolution systems \eqref{eqn:RDS} and \eqref{eqn:LDS}, respectively. In studying the reduced evolution system \eqref{eqn:LDS}, and its initial value problem, it will be instructive to regard the parameter $\delta$ as fixed and small, with the parameters $\alpha_i, \beta_i (i=1,2)$ and $D$ fixed and of $O(1)$, and $\sigma\ge \delta$ as a bifurcation parameter. Finally we remark that the above slow invariant manifold reduction of \eqref{eqn:RDS} to \eqref{eqn:LDS} when $\epsilon$ is small (and small relative to $\delta$) can be placed on a rigorous basis through the theory developed on geometric singular perturbation theory in, for example, Jones \cite{arnold1995geometric}.

\section{The Initial Value Problem (LIVP)}\label{sec_IVP}
In this section we examine preliminary results for the initial value problem for \eqref{eqn:LDS}, with the front-like initial data,
\begin{equation*}
    (M_0(x),I_0(x)) = (M_0H(-x),I_0H(-x)),~~x\in \mathbb{R}, \label{eqn3.1}
\end{equation*}
where $H(\cdot)$ is the usual Heaviside function, and $(M_0, I_0)\in R(\delta)$ are given positive constants. We will henceforth refer to this initial value problem as (LIVP). It follows from Section \ref{sec_model_reduction} that we immediately have:\\

\begin{theorem}
    Let $\delta$ be small and $\sigma\ge \delta$. Then for each positive $(M_0,I_0)\in R(\delta)$, (LIVP) has a unique, global, classical solution, $(M,I):\overline{D}_{\infty}\to R(\delta)$. Moreover, for each $T>0$,
    \begin{equation*}
    (M(x,t),I(x,t)) \to 
    \begin{cases}
        (0,0)~~\text{as}~~x\to \infty,\\
        (M_{-\infty}(t), I_{-\infty}(t))~~\text{as}~~x\to -\infty,
    \end{cases}
    \label{eqn3.2}
    \end{equation*}
    uniformly for $t\in [0,T]$, where $(M_{-\infty},I_{-\infty})\equiv (m,i):\overline{\mathbb{R}}^+\to R(\delta)$ is the unique global solution to the temporal evolution problem,
    \begin{equation}
        m' = im(1-m),~~i' = \delta^{-1} f(m,i)~~ \text{in}~~t>0, \label{eqn3.3}\tag{DS}
    \end{equation}
    subject to
    \begin{equation*}
        m(0)=M_0,~~i(0)=I_0. \label{eqn3.4}
    \end{equation*}
    In addition, $(M(x,t),I(x,t))\in (0,1)^2$ for all $(x,t)\in D_{\infty}$.
\end{theorem}
\begin{proof}
    The first part follows directly from the considerations in Section \ref{sec_model_reduction}. The second part, relating to the limits as $|x|\to \infty$, can be established following very closely the approach given in Meyer and Needham \cite{meyer2015cauchy} (see Chapter 8, Proposition 8.36 -- Remark 8.38) for a related scalar problem. The final containment is a consequence of the invariance of $R(\delta)$, followed by two respective applications of the strong parabolic maximum principle to each component equation of \eqref{eqn:LDS} on $D_{\infty}$.
\end{proof}
To quantify the limits in this result, we next consider the planar temporal dynamical system (\ref{eqn3.3})
in the invariant region $R(\delta)$ on the $(m,i)$-phase plane, and in particular when $\delta$ is small. First we locate and classify the equilibrium points of (DS) in $R(\delta)$. We immediately observe that for each $\sigma \ge \delta$, there are non-isolated equilibrium points at
\begin{equation*}
    (m,i) = (m_e,0)\equiv \mathbf{e}(m_e)~~\text{for each}~~ m_e\in [0,1-\delta]. \label{eqn3.5}
\end{equation*}
When $\sigma\in (\delta,1)$, then $\mathbf{e}(m_e)$ is a degenerate stable node for $m_e\in [0,1-\sigma)$ and a degenerate unstable node for $m_e\in (1-\sigma, 1-\delta]$. However, for $\sigma\in (1,\infty)$ then $\mathbf{e}(m_e)$ is a degenerate unstable node for each $m_e\in [0,1-\delta]$. In addition, for each $\sigma\ge \delta$, there is the fully saturated equilibrium point at
\begin{equation*}
    (m,i) = (1, b(\sigma)\sigma) \equiv\mathbf{e}_F,\label{eqn3.6}
\end{equation*}
which is always a hyperbolic stable node. These are the equilibrium points which exist in $R(\delta)$ at every $\sigma\ge \delta$, and are the only equilibrium points in $R(\delta)$ when $\sigma\in [\delta,1]$. However, at $\sigma=1$ an equilibrium point transcritical bifurcation takes place at $(m,i)=(0,0)$, which generates a further, transitional, equilibrium point in $R(\delta)$ when $\sigma\in (1,\infty)$, at
\begin{equation*}
    (m,i) = (0,b(\sigma)(\sigma-1)) \equiv \mathbf{e}_T, \label{eqn3.7}
\end{equation*}
which is a hypebolic saddle point. The location of these equilibriun points, which are all on the boundary of $R(\delta)$, together with index theory, precludes the possibility of any periodic obits in $R(\delta)$. We also observe that $m=0$ and $m=1$ are phase paths in $R(\delta)$, whilst the $\omega$-limit set for any phase path starting in the interior of $R(\delta)$ must be a subset of
\begin{equation*}
    \Omega = \begin{cases}
        \{\mathbf{e}_F\},~~\sigma\in (1,\infty),\\
        \{\mathbf{e}_F\} \cup \{\mathbf{e}(m_e): m_e\in [0, 1-\sigma]\},~~\sigma\in [\delta,1].
    \end{cases}
    \label{eqn3.8}
\end{equation*}
We can now construct the global phase portrait for (DS) in $R(\delta)$, and, in particular, we do this for $\delta$ small. It follows from (\ref{eqn3.3}), with (\ref{eqn1.17}), that there exists a stable slow manifold, given by,
\begin{equation*}
    \mathcal{S}_{\delta} = \{(m,i)\in R(\delta): i = \text{max}(0, b(\sigma)(m + (\sigma-1))) \}, \label{eqn3.9}
\end{equation*}
onto which all phase paths contract exponentially, to within $O(\delta)$, on the time scale $t=O(\delta)$, parallel
to the $i$-axis, and thereafter the dynamics remains constrained to a $\delta$-neighbourhood of $\mathcal{S}_{\delta}$. The global phase portrait in $R(\delta)$ is now readily constructed, and is shown in Figure 1, for the cases $\sigma \in [\delta,1]$ and $\sigma\in (1,\infty)$. Specifically, for initial data $(M_0,I_0)\in R(\delta)$ {and with $I_0$ positive,} we have the cases,
\begin{itemize}
    \item For $\sigma\in [\delta,1]$ and $M_0\in [0, 1-\sigma]$ then
    \begin{equation*}
        (m(t),i(t)) \to \mathbf{e}(M_0)~~\text{as}~~t\to \infty, \label{eqn3.10}
    \end{equation*}
    monotonically, on the time scale $t=O(\delta)$.
    \item For $\sigma\in [\delta,1]$ and $M_0\in (1-\sigma, 1]$ then first $(m(t),i(t))$ rapidly approaches $\mathcal{S}_{\delta}$, monotonically, on the time scale $t=O(\delta)$, and thereafter remains on $\mathcal{S}_{\delta}$ and has,
    \begin{equation*}
        (m(t),i(t)) \to \mathbf{e}_F~~\text{as}~~t\to \infty \label{eqn3.11}
    \end{equation*}
    monotonically.
\item For $\sigma\in (1,\infty)$ and $M_0\in (0, 1]$ then first $(m(t),i(t))$ rapidly approaches $\mathcal{S}_{\delta}$, monotonically, on the time scale $t=O(\delta)$, and thereafter remains on $\mathcal{S}_{\delta}$ and has,
    \begin{equation*}
        (m(t),i(t)) \to \mathbf{e}_F~~\text{as}~~t\to \infty, \label{eqn3.12}
    \end{equation*}
    monotonically.
    \item For $\sigma\in (1,\infty)$ and $M_0 = 0$ then $m(t)\equiv0$ and,
    \begin{equation*}
        (m(t),i(t)) \to \mathbf{e}_T~~\text{as}~~t\to \infty, \label{eqn3.13}
    \end{equation*}
    monotonically and rapidly, on the time scale $t=O(\delta)$.
\end{itemize}
\begin{figure}
    
   \begin{subfigure}{0.5\textwidth}

        \centering
         \includegraphics[width=0.8\linewidth]{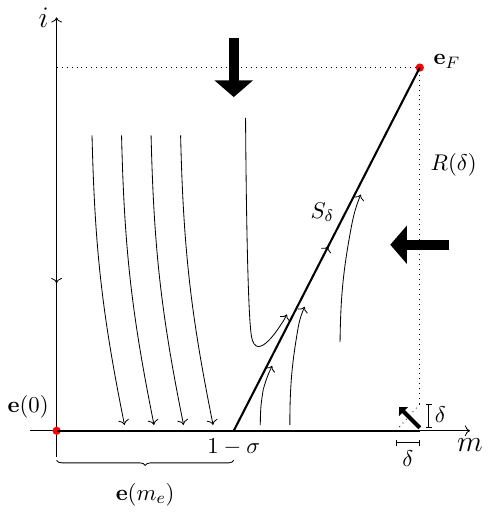}
    \caption{}
    \label{fig:enter-label}  

    \end{subfigure}
    \vspace{1em}
\raisebox{0.72em}{\begin{subfigure}{0.5\textwidth}
   
    \centering
          
          \includegraphics[width=0.8\linewidth]{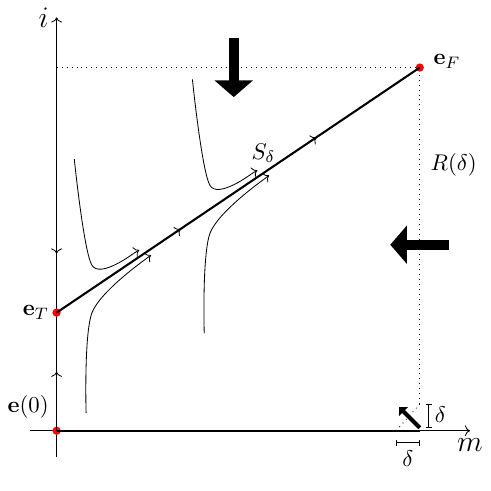}
          
    \caption{}
    \label{fig:enter-label2}  
 
    \end{subfigure}}

    \caption{
    Qualitative sketches of the phase portraits for the planar temporal dynamical system (DS), depending on $\sigma$, showing the temporal dynamics confined to the clipped region $R(\delta)$. (a) $\sigma\in[\delta,1]$; the stable slow manifold $ \mathcal{S}_{\delta}$ connects $\textbf{e}_F$ to $\textbf{e}(0)$ via the continuum of equilibrium points $\{(m_e,0): m_e\in[0,1-\sigma)]\}$. (b) $\sigma\in(1,\infty)$; the stable slow manifold connects $\textbf{e}_F$ to $\textbf{e}_T$.
    }
\end{figure}
We now have the following qualitative and quantitative results for (LIVP), when $\delta$ is small. The proof in each case can be found in Appendix \ref{app_A}.

\subsection{(LIVP) with \texorpdfstring{$\sigma\in [\delta,1-\delta^{\frac{1}{4}})$}{} and \texorpdfstring{$M_0\in (0,(1-\sigma)-\delta^{\frac{1}{4}})$}{}}

Under the above parameter restrictions, we may now state the key result concerning the qualitative structure of the solution to (LIVP),\\

\begin{theorem} \label{thm3.2}
With $\delta$ small, let $(M,I):\overline{D}_{\infty}\to R(\delta)$ be the solution to (LIVP). Then
\begin{equation*}
  0<I(x,t) < \frac{1}{2}I_0 \exp(-\delta^{-1}c(M_0,\sigma)t) \erfc(x/2(Dt)^{\frac{1}{2}}) ~~\forall~~(x,t)\in D_{\infty}, \label{eqn3.23}
\end{equation*}
and
\begin{multline}
    \frac{1}{2} M_0 \erfc(x/2t^{\frac{1}{2}}) < M(x,t) < \frac{1}{2}M_0\exp\left(I_0\frac{\delta}{c(M_0,\sigma)}(1- \exp(-\delta^{-1}c(M_0,\sigma)t))\right)\\ \erfc(x/2t^{\frac{1}{2}}) ~~\forall~~(x,t)\in D_{\infty}. \label{eqn3.24}
\end{multline}  
with
 \begin{equation}
        c(M_0,\sigma) = \frac{a(\sigma)b(\sigma)((1-\sigma)-M_0)}{2(\alpha_2 + \beta_2)}>0.
        \end{equation}
\end{theorem}
\hfill\break As a consequence of this result we have immediately,\\

\begin{corollary}\label{cor3.1}
  With $\delta$ small, let $(M,I):\overline{D}_{\infty}\to R(\delta)$ be the solution to (LIVP). Then  
  \begin{equation*}
      I(x,t) = O(I_0 \exp(-\delta^{-1}c(M_0,\sigma)t) \erfc(x/2(Dt)^{\frac{1}{2}})),
  \end{equation*}
  \begin{equation*}
     M(x,t) = \frac{1}{2}M_0(\erfc(x/2t^{\frac{1}{2}}) (1 + O(\delta)),
  \end{equation*}
  as $\delta\to 0$, uniformly for $t\ge O(1)$ and $x\in \mathbb{R}$.
\end{corollary}
\hfill\break Thus, in this case, we see that the species $I$ decays to zero exponentially fast on the rapid time scale $t=O(\delta)$, leaving, in effect, the initial distribution of species $M$ to simply diffuse on the time scale $t=O(1)$. Therefore we may conclude, in this case, that the critical value $M_0 = (1-\sigma)$ provides a \emph{threshold value} for the initial value problem (LIVP) in relation to the activation of any kind of travelling wavefront structures.

\subsection{(LIVP) with \texorpdfstring{$\sigma\in [\delta,1]$ and $M_0\in (1-\sigma, 1]$}{}, or, \texorpdfstring{$\sigma\in (1, \infty)$}{} and \texorpdfstring{$M_0\in (0,1]$}{}}
To begin with it is convenient to introduce $u=U_F: \overline{D}_{\infty} \to [0,1]$ as the global and unique solution to the classical FKPP scalar reaction-diffusion PDE, with reaction function $ u(1-u)$, subject to the initial condition $u(x,0) = H(-x)~\forall~x\in \mathbb{R}$ (see, for example, Fife \cite{Fife1979}). We then have,\\
\begin{theorem} \label{thm3.3}
  With $\delta$ small, let $(M,I):\overline{D}_{\infty}\to R(\delta)$ be the solution to (LIVP). Then,
  \begin{equation*}
 \frac{1}{2}M_0\erfc(x/t^{\frac{1}{2}}) < M(x,t) < U_F(x,t),
 \end{equation*}
 and
 \begin{equation*}
     0 < I(x,t) < b(\sigma)\sigma + \left(\frac{1}{2}I_0\erfc(x/2t^{\frac{1}{2}}) - b(\sigma)\sigma\right)\exp\left(-a(\sigma)/(\alpha_2\delta)t\right)
 \end{equation*}
 for all $(x,t)\in D_{\infty}$.
\end{theorem}

\hfill\break In this case the above result indicates the possibility of permanent form travelling wavefronts developing in both of the species $M$ and $I$ when $t$ is large. To initiate further investigation into this possible key feature, we now consider the existence and structure of permanent form propagating wavefront solutions to the system \eqref{eqn:LDS}, which are non-negative, and propagate into the zero species equilibrium state $(M,I) = (0,0)$ at a positive and constant propagation speed.
\section{Permanent Form Travelling Waves}\label{sec_PTW}
Throughout we consider $0<\delta\ll1$ and there are three cases which it is instructive to consider separately. We begin with the case when $\delta\le \sigma<1$.

\subsection{Travelling waves when \texorpdfstring{$\delta\le \sigma<1$}{}}\label{sec_4.1}
To begin we have the definition:\\

\begin{definition}
A \textbf{Full Transition Permanent Form Travelling Wave} solution to the system \eqref{eqn:LDS} (a [FPTW]) is a non-negative solution to \eqref{eqn:LDS}, say $(M_T,I_T):\mathbb{R}\to (\overline{\mathbb{R}}^+\times\overline{\mathbb{R}}^+)$, which depends only upon the travelling wave coordinate $z=x-vt$, where $v> 0$ is the constant propagation speed, and satisfies,
\begin{equation*}
    (M_T(z),I_T(z))\to
    \begin{cases}
    \mathbf{0}~~\text{as}~~z\to \infty\\
    \mathbf{e}_F~~\text{as}~~z\to -\infty
    \end{cases}
\end{equation*}
with
\begin{equation*}
(M_T(z),I_T(z)) \in [0,1]^2~~\forall~~z\in \mathbb{R}.    
\end{equation*}
\end{definition}
\hfill\break It is now straightforward to determine that [FPTW] solutions are precisely solutions to the nonlinear eigenvalue problem (with eigenvalue $v>0$),
\begin{varsubequations}{EVP1}\label{eqn:EVP1}
  \begin{equation}
    M_T'' + vM_T' + I_TM_T(1-M_T) = 0,~~z\in \mathbb{R}, \label{eqn4.3}
\end{equation}
\begin{equation}
    \overline{\delta}( I_T'' + \overline{v}I_T') + f(M_T,I_T) = 0,~~z\in \mathbb{R}, \label{eqn4.4}
\end{equation}
subject to
\begin{equation}
(M_T(z),I_T(z))\to
  \begin{cases}
    (0,0)~~\text{as}~~z\to \infty,\\
    (1,b(\sigma)\sigma)~~\text{as}~~z\to -\infty, \label{eqn4.5}
    \end{cases}
\end{equation}
\begin{equation}
 (M_T(z),I_T(z)) \in [0,1]^2~~\forall~~z\in \mathbb{R}.
\end{equation}  
\end{varsubequations}
Here, for convenience, we have introduced the notation,
\begin{equation*}
    \overline{\delta}=D\delta,~~\overline{v}=D^{-1}v.
\end{equation*}
To begin we have some basic preliminary results (whose proofs are deferred to Appendix \ref{app_B}):\\
\begin{prop}\label{prop_4.1}
    Let $(M_T,I_T):\mathbb{R}\to [0,1]^2$ be a [FPTW], then $M_T(z)>0$ and $I_T(z)>0$ for all $z\in \mathbb{R}$.
\end{prop}

\hfill\break Next we have upper bounds:\\
\begin{prop}
  Let $(M_T,I_T):\mathbb{R}\to [0,1]^2$ be a [FPTW], then $M_T(z)<1$ and $I_T(z)<b(\sigma)\sigma$ for all $z\in \mathbb{R}$.   
\end{prop}

\hfill\break and monotonicity:\\
\begin{prop}\label{prop_4.3}
  Let $(M_T,I_T):\mathbb{R}\to [0,1]^2$ be a [FPTW]. Then both $M_T(z)$ and $I_T(z)$ are strictly monotone decreasing for all $z\in \mathbb{R}$.   
\end{prop}
\hfill\break

\begin{remk}
    This result allows us to uniquely fix translation invariance in \eqref{eqn:EVP1} by requiring that
    \begin{equation*}
        M_T(0) = 1-\sigma, \label{eqn4.24}
    \end{equation*}
    which we adopt henceforth in this subsection.
\end{remk}

\hfill\break We now consider the existence and structure of [FPTW] solutions in the limit $\overline{\delta} \to 0$. It is convenient to use a hodograph type of approach. To this end let $(M_T,I_T):\mathbb{R}\to [0,1]^2$ be a [FPTW] solution, then using Proposition \ref{prop_4.3} we may write
\begin{equation*}
    I_T(z) = I_T(z(M_T))\equiv H(M_T)~~\text{for}~~M_T\in [0,1],  \label{eqn4.25}
\end{equation*}
where $z(M_T):[0,1]\to \mathbb{R}$ is the inverse function of $M_T(z):\mathbb{R}\to [0,1]$. It follows that
\begin{equation*}
    I_T'(z) = H'(M_T)/z'(M_T),
    \quad
    I''_T(z) = H''(M_T)/(z'(M_T))^2 - z''(M_T)H'(M_T)(z'(M_T))^3.
\end{equation*}
Thus equation (\ref{eqn4.4}) becomes,
\begin{multline}
    \overline{\delta}\left(H''(M_T)/(z'(M_T))^2 - z''(M_T)H'(M_T)(z'(M_T))^3 + \overline{v} H'(M_T)/z'(M_T)\right)\\ + \frac{a(\sigma)H(M_T)(b(\sigma)M_T - b(\sigma)(1-\sigma)-H(M_T))}{(\alpha_2H(M_T) + \beta_2(1-M_T))} = 0,~~M_T\in (0,1). \label{eqn4.28}
\end{multline}
We require $H(M_T)\in (0,b(\sigma)\sigma)$ for $M_T\in (0,1)$, via Propositions 4.1--4.3, and,
\begin{equation}
    H(0)=0,~~H(1)=b(\sigma)\sigma, \label{eqn4.29}
\end{equation}
with 
\begin{equation}
    H(M_T)~\text{increasing with}~~M_T\in (0,1). \label{eqn4.30}
\end{equation}
We now write,
\begin{equation}
    H(M_T;\overline{\delta}) = H_0(M_T) + o(1)~~\text{as}~~\overline{\delta}\to 0,  \label{eqn4.31}
\end{equation}
with $M_T\in [0,1]$. At leading order in equation (\ref{eqn4.28}), with conditions (\ref{eqn4.29}) and (\ref{eqn4.30}), we obtain,
\begin{equation}
    H_0(M_T) =
    \begin{cases}
        b(\sigma)(M_T - (1-\sigma)),~M_T\in ((1-\sigma),1]\\
        0,~M_T\in [0,(1-\sigma)].
    \end{cases}
    \label{eqn4.32}
\end{equation}
We note that there will be a thin transition region when
\begin{equation*}
    M_T = (1-\sigma) \pm o(1)~~\text{as}~~\overline{\delta}\to 0, \label{eqn4.33}
\end{equation*}
which will locally smooth out the gradient discontinuity in $H_0(M_T)$ at $M_T=(1-\sigma)$, and can only be considered after the leading order term in $M_T(z;\overline{\delta})$ has been determined. Thus, we now write,
\begin{equation*}
 M_T(z;\overline{\delta}) = \tilde{M}_0(z) + o(1)~~\text{as}~~\overline{\delta}\to 0,  \label{eqn4.34}   
\end{equation*}
with $z\in \mathbb{R}$. Equation (\ref{eqn4.3}) becomes, at leading order,
\begin{varsubequations}{MVP1}\label{eqn:MVP1}
    \begin{equation}
    \tilde{M}_0'' + v\tilde{M}_0' + F(\tilde{M}_0) = 0,~~z\in \mathbb{R},
\end{equation}
subject to,
\begin{equation}
    \tilde{M}_0(z)\in (0,1)~~\forall~~z\in \mathbb{R},
\end{equation}
\begin{equation}
    \tilde{M}_0(z)\to
    \begin{cases}
        1~~\text{as}~~z\to -\infty,\\
        0~~\text{as}~~z\to \infty,
    \end{cases}
\end{equation}
\begin{equation}
    \tilde{M}_0(0) = (1-\sigma).
\end{equation}
\end{varsubequations}
Here the reaction function $F\in \text{Lip}([0,1])\cap PC^1([0,1])$ is given by
\begin{equation*}
    F(X) = XH_0(X)(1-X)~~\forall~~X\in [0,1].  \label{eqn4.39}
\end{equation*}
For $\sigma\in (0,1)$ the problem \eqref{eqn:MVP1} is of classical FKPP-type with continuous cut-off in the reaction function at $\tilde{M}_0 = (1-\sigma)$. This can be readily analysed following directly the theory developed by some of the authors in Tisbury et al. \cite{TNT}, and details need not be repeated here. We can summarise these outcomes in the following:\\

\begin{theorem}\label{thm_4.1} \textbf{Summary for \eqref{eqn:MVP1}}\\
The problem \eqref{eqn:MVP1} has a solution, say $\tilde{M}_0^T:\mathbb{R}\to (0,1)$, if and only if $v=v^*(\sigma)>0$. This solution is unique, and monotone decreasing with $z\in \mathbb{R}$. In addition $v^*(\sigma)$ is continuous and increasing for $\sigma\in (0,1)$, with,
\begin{equation}
    v^*(\sigma)=
    \begin{cases}
        O(\sigma^{\frac{3}{2}})~~\text{as}~~\sigma\to 0^+,\\
        \sqrt{b(1)/2} - o(1)~~\text{as}~~\sigma\to 1^{-},
    \end{cases} \label{eqnC}
\end{equation}
whilst,
\begin{equation}
    \tilde{M}_0^T(z,\sigma) = (1-\sigma)\exp(-v^*(\sigma)z)~~\forall~~z\ge 0. \label{eqn4.41}
\end{equation}
\end{theorem}
\begin{proof}
    This follows directly the development in \cite{TNT}. 
\end{proof}

\hfill\break Using  matched asymptotic expansions (see Appendix \ref{app_C} for details) 
we develop the detailed asymptotic structure to   problem \eqref{eqn:MVP1} which yields 
  \begin{figure}[t]
      \centering
      \includegraphics[width=0.5\linewidth]{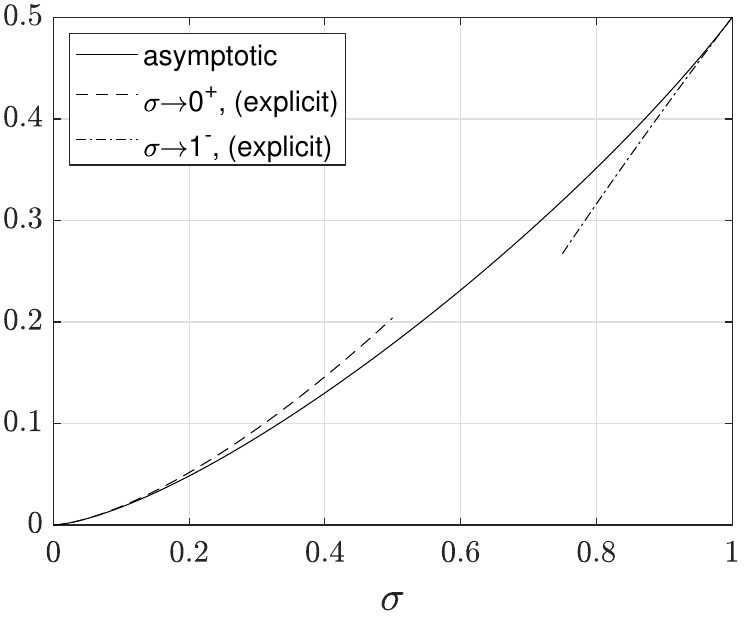}
      \caption{Propagation speed $v^{*}(\sigma)$ obtained   by solving \eqref{eqn:EVP1} numerically via a shooting method (solid line) and   the explicit asymptotic expressions provided in \eqref{eqn:vstar_sigma_asym} for $v^{*}(\sigma)$ (dashed lines). Here, $\alpha_2=\beta_2=1$ so that $b(\sigma)=(1+\sigma)^{-1}$.}
      \label{fig:NumericalComparison1D}
  \end{figure}
\begin{equation}\label{eqn:vstar_sigma_asym}
        v^*(\sigma) =
        \begin{cases}
        \sqrt{b(\sigma)/3}\sigma^{3/2}+O(b(\sigma)^{1/2}\sigma^{2})~~\text{as}~~\sigma\to 0^+,\\
b(\sigma)^\frac{1}{2}\left(\frac{1}{\sqrt{2}}-\sqrt{2}(1-\sigma)\right)+O(b(\sigma)^{1/2}(1-\sigma)^2)\text{ as }\sigma\rightarrow{1}^{-}
        \end{cases}
    \end{equation}
Figure \ref{fig:NumericalComparison1D} compares numerical results for $v^*(\sigma)$ (obtained via the numerical solution of \eqref{eqn:MVP1} employing a shooting method; see \cite{TNT} for more details) against 
the explicit asymptotic expressions in (\ref{eqn:vstar_sigma_asym}) for the particular case of $\alpha_2=\beta_2=1$ corresponding to $b(\sigma)=(1+\sigma)^{-1}$.
The agreement between the numerical results and the explicit asymptotic expressions is excellent as $\sigma\to 0^+$ and $\sigma\to 1^-$.

Henceforth we set $v=v^*(\sigma)$ and now return to the transition region for $H(M_T;\overline{\delta})$. In the transition region we write,
\begin{equation}
    M_T = (1-\sigma) + \overline{\delta}^r m , \label{eqn4.42}
\end{equation}
with $m=O(1)$, and $r>0$ to be determined. It follows from (\ref{eqn4.31}) and (\ref{eqn4.32}) that we have,
\begin{equation}
H(M_T;\overline{\delta}) = O(\overline{\delta}^r),    
\end{equation}
as $\overline{\delta}\to 0$ in the transition region. Thus we write,
\begin{equation}
H(M_T;\overline{\delta}) =\overline{H}(m) \overline{\delta}^r + o(\overline{\delta}^r). \label{eqn4.42'}   
\end{equation}
Also, via (\ref{eqn4.41}) and (\ref{eqn4.42}), we have,
\begin{equation*}
    z_T(m) = - \frac{m}{(1-\sigma)v^*(\sigma)} \overline{\delta}^r + o(\overline{\delta}^r),
\end{equation*}
which gives,
\begin{equation*}
    z_T'(m) = - \frac{1}{(1-\sigma)v^*(\sigma)} + o(1)~\text{and}~z_T''(m) = o(1).
\end{equation*}
A nontrivial balance in equation (\ref{eqn4.28}) then requires,
\begin{equation}
    r = 1/3, \label{eqn4.42''}
\end{equation}
after which the leading order equation becomes,
\begin{varsubequations}{TBP}\label{TBP}
\begin{equation}
    \overline{H}'' + c(\sigma) \overline{H} (b(\sigma)m - \overline{H}) = 0,~~m\in \mathbb{R}, \label{eqn4.50}\tag{TBPa}
\end{equation}
which is subject to the matching conditions,
\begin{equation}
  \overline{H}(m)\sim b(\sigma)m~~\text{as}~~m\to \infty,\label{eqn4.51}\tag{TBPb}
\end{equation}
\begin{equation}
    \overline{H}(m)\to 0~~\text{as}~~m\to -\infty,\label{eqn4.52}\tag{TBPc}
\end{equation}
with
\begin{equation}
    \overline{H}(m)>0~~\forall~~m\in \mathbb{R}.\label{eqn4.53}\tag{TBPd}
\end{equation}
\end{varsubequations}
In the above,
\begin{equation*}
    c(\sigma) = \frac{a(\sigma)}{\beta_2 \sigma(1-\sigma)^2 v^{*2}(\sigma)}>0.
\end{equation*}
We first develop the boundary condition (\ref{eqn4.51}). Without giving details, for brevity, we obtain,
\begin{equation}
 \overline{H}(m) \sim b(\sigma)m + c_{\infty}(\sigma)\text{A}_{i}((c(\sigma)b(\sigma))^{\frac{1}{3}}m)~~\text{as}~~m\to \infty, \label{eqn4.54} 
\end{equation}
with $c_{\infty}(\sigma)$ a globally determined constant and $A_i(\cdot)$ the usual Airy function. Similarly developing the boundary condition (\ref{eqn4.52}) we obtain,
\begin{equation}
    \overline{H}(m) \sim c_{-\infty}(\sigma)\text{A}_{i}(-(c(\sigma)b(\sigma))^{\frac{1}{3}}m)~~\text{as}~~m\to -\infty, \label{eqn4.55}
\end{equation}
with $c_{-\infty}(\sigma)$ a positive globally determined constant. It follows from  symmetry in the problem that $c_{\infty}(\sigma)=c_{-\infty}(\sigma)$. We can use the form (\ref{eqn4.55}) to develop a shooting method for (TBP) using $c_{-\infty}$ as the shooting parameter. A numerical implementation of this method demonstrates that, for each $\sigma\in (0,1)$, \eqref{TBP} has a unique solution, and this solution is monotone increasing in $m$ (see the end of Section \ref{sec_numerics}). 

We finally move into the region where 
\begin{equation*}
    M_T\in [0,(1-\sigma) - O(\overline{\delta}^{\frac{1}{3}})),
\end{equation*}
 when it follows from (\ref{eqn4.42'}), (\ref{eqn4.42''}) and (\ref{eqn4.55}) (with $m\gg1$) that we must write,
 \begin{equation}
     H(M_T;\overline{\delta}) = \exp\left(-\frac{\Phi(M_T;\overline{\delta})}{\sqrt{\overline{\delta}}}\right), \label{4.55'}
 \end{equation}
and expand in the form,
\begin{equation}
 \Phi(M_T;\overline{\delta}) = \Phi_0(M_T) + O(\overline{\delta}\text{log}\overline{\delta})~~\text{as}~~\overline{\delta}\to 0, \label{4.55''}  
\end{equation}
with $M_T\in [0,(1-\sigma) - O(\overline{\delta}^{\frac{1}{3}}))$. Substitution in equation (\ref{eqn4.28}) gives, at leading order,
\begin{equation}
    (\Phi_0')^2 = \frac{a(\sigma)b(\sigma)((1-\sigma) - M_T)}{v^{*}(\sigma)^2\beta_2M_T^2(1-M_T)},~~M_T\in (0,(1-\sigma)). \label{eqn4.60}
\end{equation}
The solution to this equation, subject to matching with the expansion in the transition region as $M_T\to (1-\sigma)$, is readily obtained as,
\begin{equation}
    \Phi_0(M_T) = (1-\sigma)\sqrt{(\sigma c(\sigma))}\int_{M_T}^{(1-\sigma)}{\frac{\sqrt{((1-\sigma)-s)}}{s(1-s)}}ds \label{4.55'''}
\end{equation}
for $M_T\in(0,(1-\sigma))$. We readily determine from this that,
\begin{equation}
    \Phi_0(M_T) \sim 
    \begin{cases}
        \frac{2}{3}\sqrt{c(\sigma)} ((1-\sigma) - M_T)^{\frac{3}{2}}~~\text{as}~~M_T\to (1-\sigma)^-,\\
        \sqrt{\sigma(1-\sigma)^3c(\sigma)}|\text{log}(M_T)|~~\text{as}~~M_T\to 0^+.
    \end{cases} \label{4.55''''}
\end{equation}
We note that the lower limit gives, via (\ref{4.55'}),(\ref{4.55''}), (\ref{4.55'''}) and (\ref{4.55''''}), 
\begin{equation*}
    H(M_T;\overline{\delta}) \sim M_T^{\Gamma(\sigma)/{\overline{\delta}}^{\frac{1}{2}}}~~\text{as}~~M_T\to 0^+,
\end{equation*}
with $\Gamma(\sigma) = \sqrt{\sigma(1-\sigma)^3c(\sigma)}$.

We can now put the leading order forms in each region together to obtain the uniform approximation,
\begin{equation}
    H(M_T;\overline{\delta})\sim
    \begin{cases}
     b(\sigma)(M_T - (1-\sigma)),~M_T\in ((1-\sigma) + O(\overline{\delta}^{\frac{1}{3}}),1],\\
     \overline{H}((M_T - (1-\sigma))\overline{\delta}^{-\frac{1}{3}}) \overline{\delta}^{\frac{1}{3}},~M_T\in ((1-\sigma) - O(\overline{\delta}^{\frac{1}{3}}),(1-\sigma) + O(\overline{\delta}^{\frac{1}{3}})),\\
     \exp\left(-\frac{\Phi(M_T;\overline{\delta})}{\sqrt{\overline{\delta}}}\right),~M_T\in [0,(1-\sigma) - O(\overline{\delta}^{\frac{1}{3}})), \label{eqn4.64}
    \end{cases}
\end{equation}
as $\overline{\delta}\to 0$, uniformly for $M_T\in [0,1]$. Thus, finally we have
\begin{equation}
    I_T(z;\overline{\delta}) \sim H(\tilde{M}_0^T(z);\overline{\delta}), \label{eqn4.65}
\end{equation}
\begin{equation}
    M_T(z;\overline{\delta})\sim \tilde{M}_0^T(z), \label{eqn4.66}
\end{equation}
as $\overline{\delta}\to 0$, uniformly for $z\in \mathbb{R}$, and with $v\sim v^*(\sigma)$ as $\overline{\delta} \to 0$. It should be noted from (\ref{eqn4.64})--(\ref{eqn4.66}) that, whilst,
\begin{equation*}
  M_T(z;\overline{\delta}) = O(1),~~z\in \mathbb{R},   
\end{equation*}
we have,
\begin{equation*}
 I_T(z;\overline{\delta}) = 
 \begin{cases}
     O(1),~~z\in (-\infty,-O(\overline{\delta}^{\frac{1}{3}})),\\
     O(\overline{\delta}^{\frac{1}{3}}),~~z\in(-O(\overline{\delta}^{\frac{1}{3}}),O(\overline{\delta}^{\frac{1}{3}})),\\
     O(E(\overline{\delta})),~~z\in (O(\overline{\delta}^{\frac{1}{3}}),\infty),
 \end{cases}
\end{equation*}
as $\overline{\delta} \to 0$, with $E(\overline{\delta})$ denoting terms which are exponentially small in $\overline{\delta}$. In fact, we have established, via the method of matched asymptotic expansions, the following (with notation as introduced earlier in this subsection):\\
\begin{theorem}
    For each $\delta$ sufficiently small and $\sigma\in [\delta,1)$, there exists a unique [FPTW], and this has propagation speed $v\sim v^*(\sigma)$ as $\delta\to 0$. The [FPTW] is given by $M_T\sim \tilde{M}_0^T(z,\sigma)$ and $I_T\sim H(\tilde{M}_0^T(z,\sigma);\overline{\delta})$ for all $z\in \mathbb{R}$, as $\delta\to 0$. Here $v^*(\sigma)$ is defined and continuous for $\sigma\in (0,1)$, is increasing with $\sigma$, and has limits,
    \begin{equation*}
        v^*(\sigma) =
        \begin{cases}
            O(\sigma^{\frac{3}{2}})~~\text{as}~~\sigma\to 0^+,\\
            \sqrt{b(1)/2}- o(1)~~\text{as}~~\sigma\to 1^-.
        \end{cases}
    \end{equation*}
    In addition,
    \begin{equation*}
        \tilde{M}_0^T(z,\sigma) = (1-\sigma)\exp(-v^*(\sigma)z)~~\forall~~z\ge0.
    \end{equation*}
\end{theorem}
\hfill\break
\begin{remk}
 We observe that the propagation speed and front thickness scales for these [FPTW] solutions are both $O(1)$ in $\delta$, as $\delta\to 0$.
\end{remk}

\hfill\break We now move on to consider the case when $\sigma=1$.

\subsection{Travelling waves when \texorpdfstring{$\sigma=1$}{}}
Again we consider [FPTW], as specified in Definition 4.1, which again satisfy \eqref{eqn:EVP1}, now with $\sigma=1$. In this case there are the simplifications,
\begin{equation*}
    \mathbf{e}_F = (1,b(1)),
\end{equation*}
\begin{equation*}
  f(M_T,I_T) = \frac{a(1)I_T(b(1)M_T - I_T)}{(\alpha_2I_T + \beta_2(1-M_T))}. 
\end{equation*}
It is also straightforward to conclude that Proposition \ref{prop_4.1}--Proposition \ref{prop_4.3} continue to hold in this case. We now consider the existence and structure of [FPTW] solutions in the limit $\overline{\delta} \to 0$, and again follow the development in the last subsection; however we now fix translation invariance by setting $M_T(0) = 1/2$. Without repeating details, we first obtain (with notation as in the previous subsection), in this case, 
\begin{equation}
    H(M_T;\overline{\delta}) \sim b(1)M_T~~\text{as}~~\overline{\delta}\to 0,
\end{equation}
uniformly for $M_T\in [0,1]$, and then
\begin{equation}
    M_T(z;\overline{\delta}) \sim \tilde{M}_0(z)~~\text{as}~~\overline{\delta}\to 0,
\end{equation}
uniformly for $z\in \mathbb{R}$. Here $\tilde{M}_0:\mathbb{R} \to \mathbb{R}$ now satisfies the problem,
\begin{varsubequations}{MVP2}\label{eqn:MVP2}
    \begin{equation}
    \tilde{M}_0'' +v \tilde{M}_0' + b(1)\tilde{M}_0^2(1-\tilde{M}_0) = 0,~~z\in \mathbb{R},\label{MVP2a}
\end{equation}
subject to,
\begin{equation}
    \tilde{M}_0(z)\to
    \begin{cases}
        1~~\text{as}~~z\to -\infty,\\
        0~~\text{as}~~z\to \infty,
    \end{cases}\label{MVP2b}
\end{equation}
\begin{equation}
    \tilde{M}_0(z)\in (0,1)~~\forall~~z\in \mathbb{R},\label{MVP2c}
\end{equation}
\begin{equation}
    \tilde{M}_0(0) = 1/2,\label{MVP2d}
\end{equation}
\end{varsubequations}

This problem, often referred to as the cubic Fisher problem, has been studied in detail in \cite{billingham1991note}, and we need only summarise the results here. These are as follows:\\
\begin{theorem}\label{thm_4.3} \textbf{Summary for \eqref{eqn:MVP2}}\\
For each $v\ge v^*(1)=\sqrt{b(1)/2}$ the problem \eqref{eqn:MVP2} has a solution $\tilde{M}_0^T:\mathbb{R}\to (0,1)$, which is unique. The solution is strictly monotone decreasing with $z\in \mathbb{R}$. In addition, when $v=v^*(1)$,
\begin{equation*}
    \tilde{M}_0^T(z) = \frac{\exp(-v^*(1)z)}{(1 + \exp(-v^*(1)z))}~~\forall~~z\in \mathbb{R}.
\end{equation*}
 Also,
 \begin{equation*}
     \tilde{M}_0^T(z) \sim
     \begin{cases}
       \exp(-v^*(1)z),~v=v^*(1),\\
       v/(b(1)z),~v>v^*(1),
     \end{cases}
 \end{equation*}
 as $z\to \infty$, whilst,
 \begin{equation*}
     \tilde{M}_0^T(z)\sim 1 - A_{-\infty}(v)\exp((\sqrt{v^2+4b(1)} - v)/2)z),
 \end{equation*}
 as $z\to {-\infty}$, for any $v\ge v^*(1)$, with $A_{-\infty}(v)$ being a positive and globally determined constant. For each $v\in [0,v^*(1))$ \eqref{eqn:MVP2} has no solution.
\end{theorem}
\begin{proof}
    After a simple coordinate scaling, this follows directly from Billingham and Needham \cite{billingham1991note}.
\end{proof}
\hfill\break Thus, we have established the following for \eqref{eqn:EVP1} when $\sigma=1$:\\
\begin{theorem}
    For each $\delta$ sufficiently small and $\sigma=1$, there exists a unique (up to translation) [FPTW] for each propagation speed $v\ge v_m(\delta)$, with $v_m(\delta)\sim \sqrt{b(1)/2}$ as $\delta\to 0$. This is given by $M_T\sim \tilde{M}_0^T(z)$ and $I_T\sim b(1)\tilde{M}_0^T(z)$ as $\delta\to 0$ uniformly for $z\in \mathbb{R}$. There are no [FPTW] solutions at any propagation speed $v\in [0,v_m(\delta))$. 
\end{theorem}

\hfill\break
\begin{remk}
    We again observe that the propagation speed and front thickness scales for these [FPTW] solutions are both $O(1)$ in $\delta$, as $\delta\to 0$.
\end{remk}
\hfill\break We now move to the final case with $\sigma >1$.

\subsection{Travelling waves when \texorpdfstring{$\sigma>1$}{}}
In this final case we define an [FPTW] as before. In addition we introduce:\\

 \begin{definition}
  An \textbf{Upper Transition Permanent Form Travelling Wave} solution to the system \eqref{eqn:LDS} (a [UPTW]) is a non-negative solution to \eqref{eqn:LDS}, say $(M_T,I_T):\mathbb{R}\to (\overline{\mathbb{R}}^+\times\overline{\mathbb{R}}^+)$, which depends only upon the travelling wave coordinate $z=x-vt$, where $v> 0$ is the constant propagation speed, and satisfies,
\begin{equation*}
    (M_T(z),I_T(z))\to
    \begin{cases}
    \mathbf{e}_T~~\text{as}~~z\to \infty\\
    \mathbf{e}_F~~\text{as}~~z\to -\infty
    \end{cases}
\end{equation*}
with
\begin{equation*}
(M_T(z),I_T(z)) \in [0,1]^2~~\forall~~z\in \mathbb{R}.    
\end{equation*}   
 \end{definition}
 \hfill\break together with:

 \begin{definition}
  A \textbf{Lower Transition Permanent Form Travelling Wave} solution to the system \eqref{eqn:LDS} (a [LPTW]) is a non-negative solution to \eqref{eqn:LDS}, say $(M_T,I_T):\mathbb{R}\to (\overline{\mathbb{R}}^+\times\overline{\mathbb{R}}^+)$, which depends only upon the travelling wave coordinate $z=x-vt$, where $v> 0$ is the constant propagation speed, and satisfies,
\begin{equation*}
    (M_T(z),I_T(z))\to
    \begin{cases}
    \mathbf{0}~~\text{as}~~z\to \infty\\
    \mathbf{e}_T~~\text{as}~~z\to -\infty
    \end{cases}
\end{equation*}
with
\begin{equation*}
(M_T(z),I_T(z)) \in [0,1]^2~~\forall~~z\in \mathbb{R}.    
\end{equation*}   
 \end{definition}

 \subsubsection{Existence of [LPTW]'s}
 It is now straightforward to determine that [LPTW] solutions are precisely solutions to the nonlinear eigenvalue problem (with eigenvalue $v>0$),
 \begin{varsubequations}{EVP2}\label{eqn:EVP2}
     \begin{equation}
    M_T'' + vM_T' + I_TM_T(1-M_T) = 0,~~z\in \mathbb{R}, \label{eqn4.86}
\end{equation}
\begin{equation}
    \overline{\delta}( I_T'' + \overline{v}I_T') + f(M_T,I_T) = 0,~~z\in \mathbb{R}, \label{eqn4.87}
\end{equation}
subject to
\begin{equation}
(M_T(z),I_T(z))\to
  \begin{cases}
    (0,0)~~\text{as}~~z\to \infty,\\
    (0,b(\sigma)(\sigma-1))~~\text{as}~~z\to -\infty, \label{eqn4.88}
    \end{cases}
\end{equation}
\begin{equation}
 (M_T(z),I_T(z)) \in [0,1]^2~~\forall~~z\in \mathbb{R}.  
\end{equation}
 \end{varsubequations}
To begin we have the basic preliminary result (with its proof deferred to Appendix \ref{app_B}):\\
\begin{prop}\label{prop_4.4}
    Let $(M_T,I_T):\mathbb{R}\to [0,1]^2$ be a [LPTW], then $M_T(z)=0$ and $I_T(z)>0$ for all $z\in \mathbb{R}$.
\end{prop}

\hfill\break On using Proposition \ref{prop_4.4}, then \eqref{eqn:EVP2} reduces to,
\begin{varsubequations}{MVP3}\label{eqn:MVP3}
    \begin{equation}
    J_{yy} + v'J_y + G(J) = 0,~~y\in \mathbb{R},
\end{equation}
\begin{equation}
    J(y)\to
    \begin{cases}
        1~~\text{as}~~y\to -\infty,\\
        0~~\text{as}~~y\to \infty,
    \end{cases}
\end{equation}
\begin{equation}
    J(y)\in (0,1]~~\forall~~y\in \mathbb{R}.
\end{equation}
\end{varsubequations}
Here $G:[0,\infty)\to \mathbb{R}$ is given by,
\begin{equation*}
   G(X) = \frac{X(1-X)}{(1 + \gamma(\sigma) X)}~~\forall~~X\in [0,\infty), 
\end{equation*}
and we observe that,
\begin{equation*}
    G(X)\le X~~\forall~~X\in [0,1]. \label{eqn4.98}
\end{equation*}
In addition, we have introduced,
\begin{equation*}
y=C(\sigma,\delta)z,~~J=I_T/(b(\sigma)(\sigma-1)),~~v'=\overline{v}/C(\sigma, \delta),   \label{eqn4.99}
\end{equation*}
with,
\begin{equation*}
    C(\sigma,\delta)=\sqrt{a(\sigma)b(\sigma)(\sigma-1)/(\beta_2 \overline{\delta})},~~\gamma(\sigma)=\alpha_2 b(\sigma)(\sigma-1)/\beta_2.  \label{eqn4.100}
\end{equation*}
We observe that the boundary value problem \eqref{eqn:MVP3} is
a classical FKPP normalised travelling wave problem. As such, we immediately have:\\
\begin{theorem}\label{thm_4.5} \textbf{Summary for \eqref{eqn:MVP3}}\\
For each $v'\ge 2$ the problem \eqref{eqn:MVP3} has a solution $J(\cdot,v'):\mathbb{R}\to (0,1)$, which is unique (up to translation in y). The solution is strictly monotone decreasing with $y\in \mathbb{R}$. In addition, the representative translation may be chosen so that,
 \begin{equation*}
     J(y,v') \sim
     \begin{cases}
       y\exp(-y),~v'=2,\\
       \exp(-\lambda_+(v')y),~v'>2,
     \end{cases}
 \end{equation*}
 as $y\to \infty$, whilst,
 \begin{equation*}
     J(y,v')\sim 1 - A_{-\infty}(v')\exp(\mu_-(v')y)
 \end{equation*}
 as $y\to {-\infty}$, for any $v'\ge 2$, with $A_{-\infty}(v')$ being a positive and globally determined constant. Here,
 \begin{equation*}
     \lambda_+(v') = (v' - \sqrt{(v')^2-4})/2,
\quad
 \text{and}
 \quad
     \mu_-(v') = (\sqrt{(v')^2 + 4(1+\gamma(\sigma))^{-1}} - v')/2.
 \end{equation*}
 For $v'\in [0,2)$ \eqref{eqn:MVP3} has no solutions.
\end{theorem}
\begin{proof}
   This follows directly from the classical FKPP travelling wave theory (see, for example, Volpert \cite{volpert1994traveling}).
\end{proof}
 \hfill\break We have now established:\\
 \begin{theorem}\label{thm:4.6}
     For each $\delta>0$ and $\sigma>1$, there exists a unique (up to translations in z) [LPTW] at each propagation speed $v\ge v_m^L(\sigma,\delta)$, given by $M_T(z,v)=0$ and $I_T(z,v)= I_0(\sigma)J(C(\sigma,\delta)z, v/(DC(\sigma,\delta)))$ for all $z\in \mathbb{R}$. Here,
     \begin{equation*}
         I_0(\sigma) = b(\sigma)(\sigma-1),
     \quad
     \text{and}
     \quad
      v_m^L(\sigma,\delta) = 2 \left(\sqrt{a(\sigma)b(\sigma)(\sigma-1)/\beta_2}\right) \sqrt{D}/\sqrt{\delta}.   
     \end{equation*}
     There are no [LPTW] solutions at any propagation speed $v\in [0,v_m^L(\sigma,\delta))$.
 \end{theorem}
 \hfill\break
 \begin{remk}
    We now observe that the propagation speed and front thickness scales for these [LPTW] solutions are of $O(1/\sqrt{\delta})$ and $O(\sqrt{\delta})$ in $\delta$, as $\delta\to 0$, respectively.
\end{remk}

 \subsubsection{Existence of [UPTW]'s}
  We determine that [UPTW] solutions are precisely solutions to the nonlinear eigenvalue problem (with eigenvalue $v>0$),
  \begin{varsubequations}{EVP3}\label{eqn:EVP3}
    \begin{equation}
    M_T'' + vM_T' + I_TM_T(1-M_T) = 0,~~z\in \mathbb{R}, \label{eqn4.107}  
\end{equation}
\begin{equation}
    \overline{\delta}( I_T'' + \overline{v}I_T') + f(M_T,I_T) = 0,~~z\in \mathbb{R}, \label{eqn4.108}   
\end{equation}
subject to
\begin{equation}
(M_T(z),I_T(z))\to
  \begin{cases}
    (0,b(\sigma)(\sigma-1))~~\text{as}~~z\to \infty,\\
    (1,b(\sigma)\sigma)~~\text{as}~~z\to -\infty, \label{eqn4.109}   
    \end{cases}
\end{equation}
\begin{equation}
 (M_T(z),I_T(z)) \in [0,1]^2~~\forall~~z\in \mathbb{R}. 
\end{equation}  
  \end{varsubequations}
It is straightforward to conclude that Proposition \ref{prop_4.1}--Proposition \ref{prop_4.3} continue to hold in this case for [UPTW]'s. We now consider the existence and structure of [UPTW] solutions in the limit $\overline{\delta} \to 0$, and again follow the development in the previous subsections. Equation (\ref{eqn4.108}), together with smoothness of [UPTW]'s, gives, at leading order,
\begin{equation*}
    I_T(z)\sim b(\sigma)(M_T(z) +(\sigma-1))~~\text{as}~~\overline{\delta}\to 0,  \label{eqn4.101}
\end{equation*}
uniformly for all $z\in \mathbb{R}$. On substitution into equation (\ref{eqn4.107}), after a rescaling, we arrive at the boundary value problem,
\begin{varsubequations}{MVP4}\label{eqn:MVP4}
 \begin{equation}
    M_T'' + V M_T' + M_T(1 + (\sigma-1)^{-1}M_T)(1-M_T) = 0,~~y\in \mathbb{R},
\end{equation}
subject to,
\begin{equation}
    M_T(y)\to
    \begin{cases}
        1~~\text{as}~~y\to -\infty,\\
        0~~\text{as}~~y\to \infty,
    \end{cases}
\end{equation}
\begin{equation}
    M_T(y)\in (0,1)~~\forall~~z\in \mathbb{R}.
\end{equation} 
\end{varsubequations}
Here $'=d/dy$ and,
\begin{equation*}
    y=\sqrt{b(\sigma)(\sigma-1)}z,~~V=v/\sqrt{b(\sigma)(\sigma-1)}.
\end{equation*}
Again \eqref{eqn:MVP4} is of classical FKPP type, but now without the gradient constraint on the reaction function. Following the classical FKPP travelling wave theory (see, for example, Volpert and Volpert \cite{volpert1994traveling}, and for this particular case, Leach and Needham \cite{LNBK}, Example 2.3, p. 35) we may state:\\
\begin{theorem}\label{thm_4.7} \textbf{Summary for (MVP4)}\\
For each $V\ge V_m(\sigma)$ the problem (MVP4) has a solution $M_T=M^*(\cdot,V,\sigma):\mathbb{R}\to (0,1)$, which is unique (up to translation in y). The solution is strictly monotone decreasing with $y\in \mathbb{R}$. Moreover, there exists a value $\sigma = 3/2$ such that $V_m(\sigma)>2$ for all $\sigma\in (1,3/2)$ whilst $V_m(\sigma)=2$ for all $\sigma\in [3/2,\infty)$. In fact,
\begin{equation}
    V_m(\sigma) = \sqrt{2}(\sigma-1)^{\frac{1}{2}} +  (\sqrt{2}(\sigma-1)^{\frac{1}{2}})^{-1}~~\forall~~\sigma\in (1,3/2)
\end{equation}
Also we may choose the representative translation so that, for $\sigma\in (3/2,\infty)$,
 \begin{equation}
     M^*(y,V,\sigma) \sim
     \begin{cases}
       y\exp{}(-y),~V=2,\\
       \exp{}(-\Gamma_+(V)y),~V>2,
     \end{cases}
 \end{equation}
 as $y\to \infty$, whilst, for $\sigma\in (1,3/2)$
  \begin{equation}
     M^*(y,V,\sigma) \sim
     \begin{cases}
       \exp{}(-\Gamma_-(V)y),~V=V_m(\sigma),\\
       \exp{}(-\Gamma_+(V)y),~V>V_m(\sigma),
     \end{cases}
 \end{equation}
 as $y\to \infty$, and for $\sigma= 3/2$,
 \begin{equation}
     M^*(y,V,\sigma) \sim
     \begin{cases}
       \exp{}(-y),~V=2,\\
       \exp{}(-\Gamma_+(V)y),~V>2,
     \end{cases}
 \end{equation}
 as $y\to \infty$.
 Correspondingly,
 \begin{equation}
     M^*(y,V,\sigma)\sim 1 - A_{-\infty}(V,\sigma)\exp{}(\Delta_-(V,\sigma)y)
 \end{equation}
 as $y\to {-\infty}$, for any $V\ge V_m(\sigma)$, with $A_{-\infty}(V,\sigma)$ being a positive and globally determined constant. Here,
 \begin{equation}
     \Gamma_{\pm}(V) = (V -(\pm \sqrt{V^2-4})))/2,
 \quad\text{and}\quad
     \Delta_-(V,\sigma) = (\sqrt{V^2 + 4(1+(\sigma-1)^{-1})} - V)/2.
 \end{equation}
 For $V\in [0,V_m(\sigma))$ (MVP3) has no solutions.
\end{theorem}
\hfill\break 
{\begin{remk}
    For $1<\sigma\le 3/2$, when the propagation speed takes on its minimum value $V = V_m(\sigma)$, we have the exact expression for the waveform,
    \begin{equation}
     M^*(y,V_m(\sigma),\sigma) = \left(1 + \exp{}((2(\sigma-1))^{-1/2} y)\right)^{-1}~~\forall~~y\in \mathbb{R}.   
    \end{equation}
\end{remk}}
\hfill\break Therefore, concerning [UPTW] solutions we have:\\
\begin{theorem}\label{thm:4.8}
    For each $\delta$ sufficiently small and $\sigma>1$, there exists a unique (upto translation) [UPTW] for each propagation speed $v\ge v_m^u(\sigma;\delta)$, with $v_m^u(\sigma;\delta)\sim V_m^u(\sigma)$ as $\delta\to 0$. This is given by $M_T\sim M^*(z\sqrt{b(\sigma)(\sigma-1)},v/\sqrt{b(\sigma)(\sigma-1)},\sigma)$ and $I_T\sim b(\sigma)((\sigma-1) + M^*(z\sqrt{b(\sigma)(\sigma-1)},v/\sqrt{b(\sigma)(\sigma-1)},\sigma)$ as $\delta\to 0$ uniformly for $z\in \mathbb{R}$. Here $V_m^u(\sigma)$ is continuous, and,
    \begin{equation}\label{eqn:ExplicitySpeed1}
     V_m^u(\sigma)= \sqrt{2b(\sigma)} ((\sigma-1) +  1/2)~   
    \end{equation}
    for $\sigma\in (1,3/2)$ whilst,
    \begin{equation}\label{eqn:ExplicitySpeed2}
     V_m^u(\sigma)=2\sqrt{b(\sigma)(\sigma-1)}   
    \end{equation}
    for $\sigma\in [3/2,\infty)$. 
    There are no [UPTW] solutions at any propagation speed $v\in [0,v_m^u(\sigma;\delta))$. 
\end{theorem}
\hfill\break
\begin{remk}
    We observe that the propagation speed and front thickness scales for these [UPTW] solutions are both $O(1)$ in $\delta$, as $\delta\to 0$.
\end{remk}

\subsubsection{Nonexistence of [FPTW]'s}
In relation to the existence of [FPTW] solutions we have (with the proof deferred to Appendix \ref{app_B}):\\
\begin{theorem}
    For each $\delta$ sufficiently small and $\sigma >1$ there are no [FPTW] solutions
\end{theorem}

\hfill\break 
\hfill\break To summarise this section, we have now completed the analysis of all possible non-negative families of permanent form travelling wave solutions associated with the system \eqref{eqn:LDS}. 
%
{In relation to the initial value problem when the initial data satisfies the threshold $M_0 > \text{max}(1-\sigma,0)$ we have seen above that there are PTW structures available which may develop as large-t attractors. Specifically, when $\delta\le \sigma\le 1$ we anticipate the formation of a [FPTW], whilst for $\sigma>1$ we expect the formation of the fast [LPTW] spatially preceded by the slower [UPTW]. A further distinction between the three families of travelling waves is important to make: {[LPTW] (which exist only when $\sigma>1$) are independent of mutant epithelial cells, meaning that $M_T(z)=0$ for all $z\in\mathbb{R}$, which is in contrast to [FPTW] (which exist only when $\delta<\sigma\le1$) and [UPTW] (which exist only when $\sigma>1$) where $M_T(z)>0$ for all $z\in\mathbb{R}$.}

\section{Numerics}\label{sec_numerics}
{In the last section we have formulated the problems (\ref{eqn:EVP1}), (\ref{eqn:EVP2}) and (\ref{eqn:EVP3})  for [FPTW], [UPTW] and [LPTW] travelling wave solutions respectively, associated with the initial value problem (LIVP). In particular a full theory was developed with the parameter $\delta$ fixed and small and $\sigma\ge \delta$ as a bifurcation parameter. We now obtain numerical approximations $(M_T(z),I_T(z))$ for each type of travelling wave by considering careful numerical solutions to: (i) the approximate problems  (\ref{eqn:MVP1}), (\ref{eqn:MVP2}), (\ref{eqn:MVP3})
and (\ref{eqn:MVP4}) in support of the theory of the last section; (ii) the exact problems (\ref{eqn:EVP1}), (\ref{eqn:EVP2}) and (\ref{eqn:EVP3}), both when $\delta$ is small (for comparison with the theory of the last section) and with an increasing range of values of $\delta$ (in both cases representative values of $\sigma$ are taken)} 

For (i) {we simply seek to supplement the detailed information provided in Theorems \ref{thm_4.1}, \ref{thm_4.3}, \ref{thm_4.5} and \ref{thm_4.7}.  For $\sigma=1$ all the required information to (\ref{eqn:MVP2})  is supplied in Theorem \ref{thm_4.3}, and no numerical calculations need to be made. However, for $\sigma\in [\delta,1) \cup (1,\infty)$ the details of Theorems \ref{thm_4.1}, \ref{thm_4.5} and \ref{thm_4.5} can be illuminated by numerical solution of (\ref{eqn:MVP1}), (\ref{eqn:MVP2}) 
and (\ref{eqn:MVP4}) respectively. 
These numerical solutions are obtained by moving to the $(\chi,\chi')$ phase plane (where here $\chi$ represents $\tilde{M}_0, J$ or $M_T$ when considering (\ref{eqn:MVP1}), (\ref{eqn:MVP2}) 
and (\ref{eqn:MVP4}) respectively) and using MATLAB's initial value solver \texttt{ode45} together with a  shooting-type algorithm, relying on a linear approximation of the unstable manifold near $(\chi,\chi')=(1,0)$.} This is parametrised by the value of $v$ which is \textit{a priori} unknown when $\sigma<1$, but is known when $\sigma>1$ (for further details, see \cite{TNT}).
In all cases, we find it sufficient to prescribe both absolute and relative error tolerances of $10^{-13}$.
For (ii), when we are solving (\ref{eqn:EVP1}), (\ref{eqn:EVP2}) or (\ref{eqn:EVP3}) directly, and with $\delta$, not necessarily asymptotically small, we instead use MATLAB's boundary value problem solver \texttt{bvp5c}, which is seeded with the asymptotic approximation as the initial guess, and solved for $z\in[-L,L]$, where $L=30$ on a uniform grid of $N=5000$ grid points. We prescribe the same absolute and relative tolerances as in (i). The boundary conditions are once again based on  a linear approximation of the relevant manifolds near the  
equilibrium points. 
For all values of $\sigma$ we take the boundary condition $\mathbf{U}_T'(-L)=\lambda_F\mathbf{U}(-L)$ where $\mathbf{U}_T=(M_T,I_T)^T$ and $\lambda_{F}$ is the eigenvalue associated to the unstable manifold of the equilibrium point $\mathbf{e}_F$.
For $\sigma<1$ we take the additional boundary condition $\mathbf{U}_T'(L)=\text{diag}(\lambda_0,\lambda_1)\mathbf{U}_T(L)$,
where $\lambda_0$ and $\lambda_1$ are the leading and next order eigenvalues associated to   the stable manifold of the equilibrium point $\mathbf{e}(0)$. 
For $\sigma\geq{1}$ we {instead} use the boundary condition $\mathbf{U}_{T}'(L)=\lambda_T\mathbf{U}_T(L)${, where $\lambda_T$ is associated to the stable manifold of the equilbrium point $\mathbf{e}_T$}. The PTW propagation speed is in this case not unique and {the minimal speed $v_m$} is determined by solving for incrementally smaller speeds  until \texttt{bvp5c} returns a  non-monotonic solution.

\begin{figure}[t]
    \hspace{-1.5cm}\centering
    \includegraphics[width=0.5\linewidth]{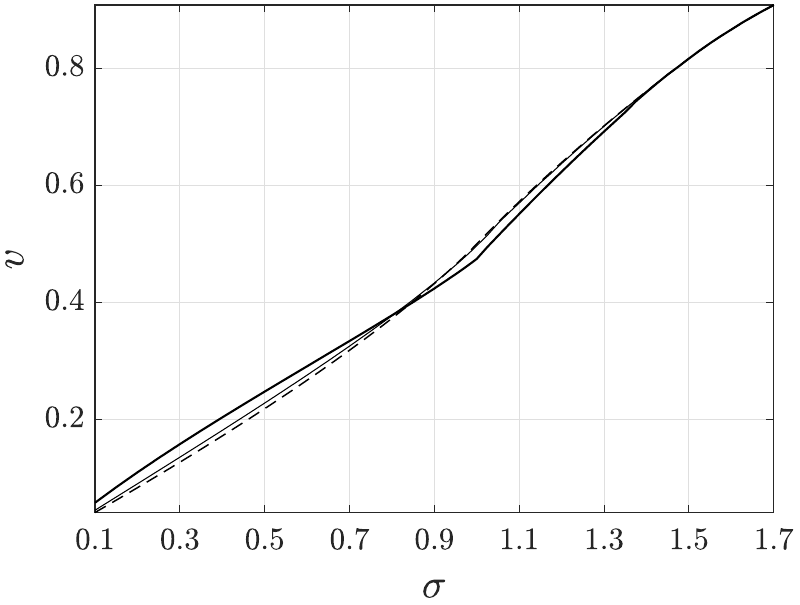}
    \caption{Numerically computed PTW =(minimum) propagation speed of the full system   obtained as a function of $\sigma$ for  $\delta=0.5$ (thick solid) and $\delta=0.05$ ({thin solid})
    by solving the set of boundary value problems (\ref{eqn:EVP1}), (\ref{eqn:EVP2}), and (\ref{eqn:EVP3}).
    This is compared against the asymptotic (dashed)  speed of propagation obtained from the one-dimensional boundary value problems  (\ref{eqn:MVP1}), (\ref{eqn:MVP2}) 
and (\ref{eqn:MVP4}). }
    \label{fig:AsymptoticVsNumeric}
\end{figure}

\begin{figure*}[t]
 \centering
        \subfloat[ $\sigma=0.75$]{%
            \includegraphics[width=0.4\linewidth]{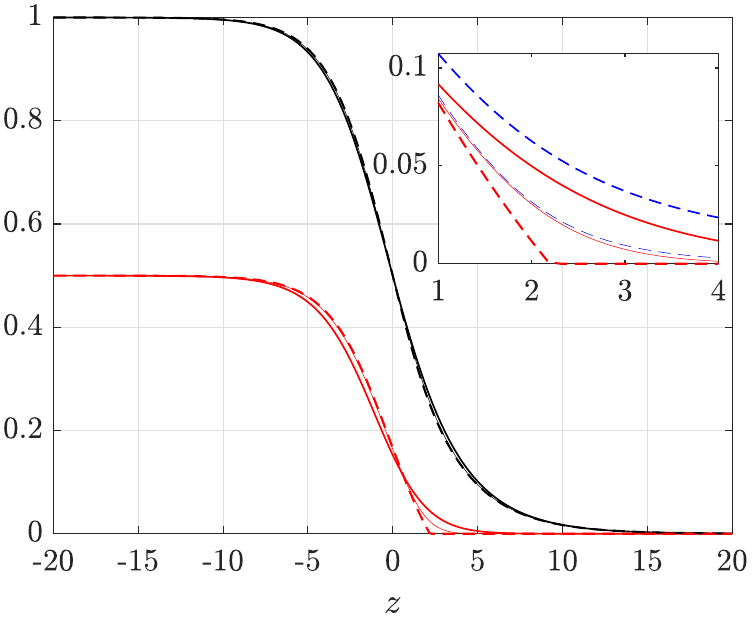}%
            \label{subfig:a}%
        }\hfill
        \subfloat[$\sigma=1$]{%
            \includegraphics[width=0.4\linewidth]{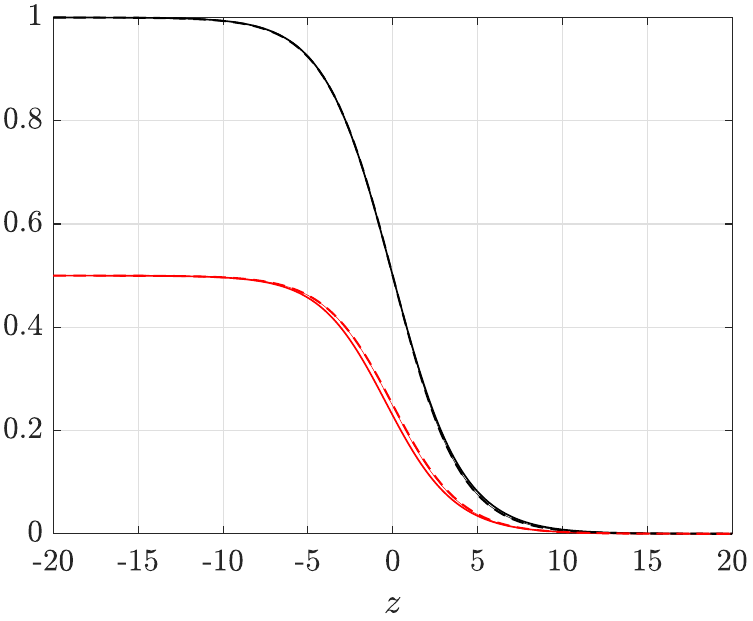}%
            \label{subfig:b}%
        }\\
        
        \subfloat[ $\sigma=1.25$]{%
            \includegraphics[width=0.4\linewidth]{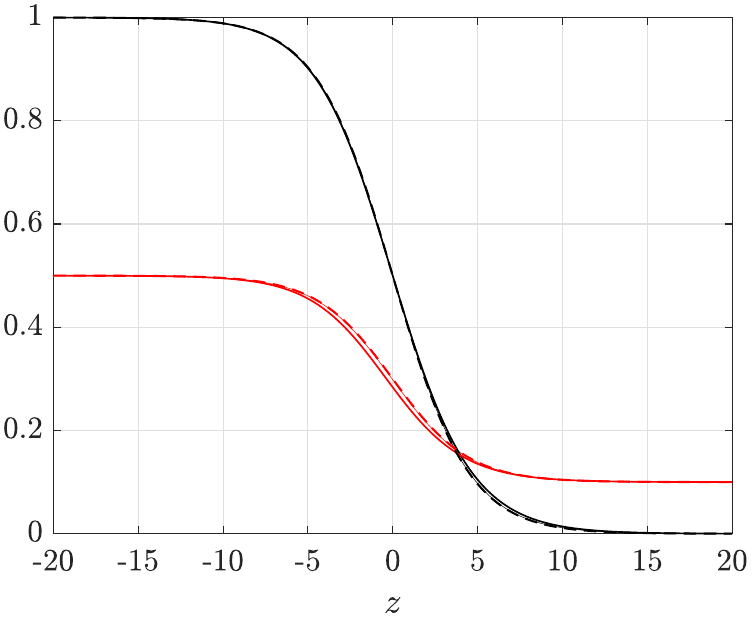}%
            \label{subfig:c}%
        }\hfill
        \subfloat[$\sigma=4$ ]{%
            \includegraphics[width=0.4\linewidth]{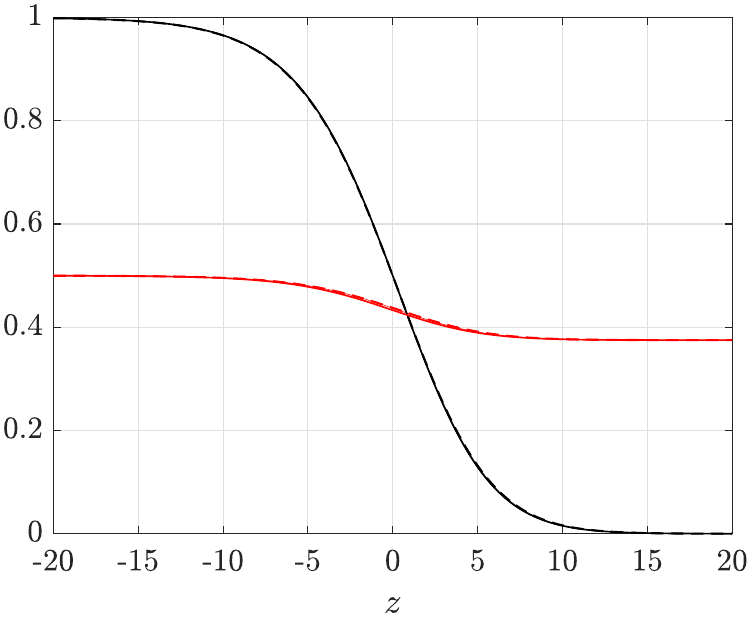}%
            \label{subfig:d}%
        }
        \caption{Numerically determined PTW profiles for $M_T(z)$ (black solid line) and $I_T(z)$ (red solid line) obtained 
        by solving the set of boundary value problems (\ref{eqn:EVP1}), (\ref{eqn:EVP2}), and (\ref{eqn:EVP3}) and
        plotted as a function of $z$ for selected values of $\sigma$, and   $\delta=0.05$ (thin solid line),  $\delta=0.5$ (thick solid line).  
These are compared against the  corresponding leading order asymptotic approximations   for $M_T(z)$ (black dashed line) and $I_T(z)$ (red dashed line) using (\ref{eqn:MVP1}) for (a),   (\ref{eqn:MVP2}) for (b) and (\ref{eqn:MVP4}) for (c) and (d). The inset in (a) focuses on the transition region near $M_T(z)=1-\sigma$ and additionally shows the first order correction to the  asymptotic approximation for $I_T(z)$  (blue dashed line) in the transition region near $M_T(z)=1-\sigma$ for the two values of $\delta$.}
        \label{fig:Profiles}
    \end{figure*}
        \begin{figure}[t]
    \hspace{-1.5cm}\centering
    \includegraphics[width=0.5\linewidth]{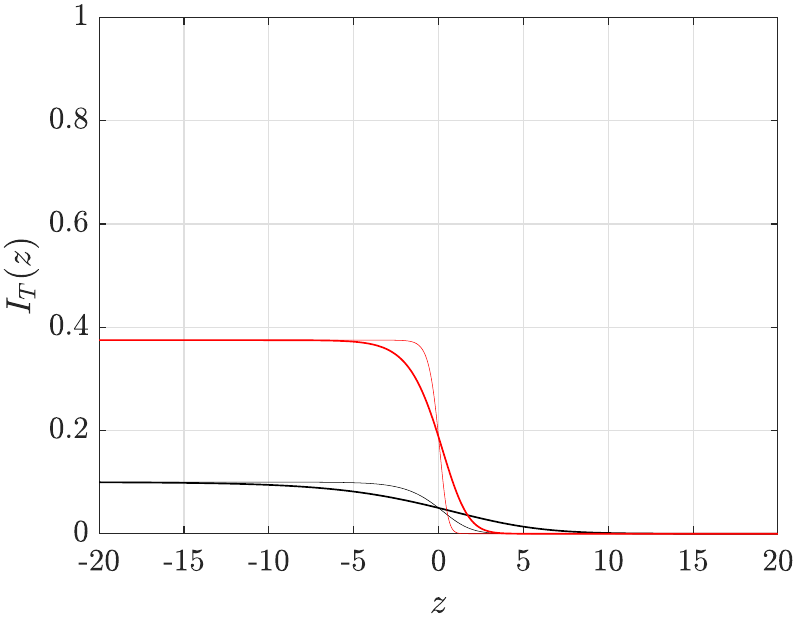}
    \caption{LPTW profiles $I_T(z)$ computed for     $\sigma=1.25$ (black lines) and $\sigma=4$ (red lines) and 
    $\delta=0.05$ (thin lines),  $\delta=0.5$ (thick lines). As $\sigma$ increases, the width of front decreases, and its maximum height increases.
    }
    \label{fig:LPTWs}
\end{figure}
Figure \ref{fig:AsymptoticVsNumeric} illustrates the PTW (mimimum) propagation speed as function of $\sigma$ from both the numerical solution of the exact problems and the asymptotic theory for small $\delta$ and shows excellent agreement that improves for decreasing $\delta$, as expected. For $\sigma>1$, the agreement   improves as $\sigma$ increases,  and we note that that the minimum propagation speed becomes independent of $\delta$ for $\sigma>3/2$. This is because in this case (\ref{eqn:EVP3}) has a non-degenerate linearisation for large $z$, and the equation for $M$ then requires (for any $\delta>0$ and $\sigma>3/2$), that minimal propagation speed $v_m\geq{}2\sqrt{b(\sigma)(\sigma-1)}$. The non-degenerate linearisation would imply (as for the classical Fisher problem) that $v_m=2\sqrt{b(\sigma)(\sigma-1)}$.

Figure \ref{fig:Profiles} shows  the associated PTW profiles 
from both the numerical solution of the exact problems and the asymptotic theory for small $\delta$
for representative values of $\sigma$.
 Again, agreement is excellent, particularly for the smaller value of $\delta$, and for $\sigma\geq{1}$.  When $\sigma>3/2$ (see Figure \ref{subfig:d}), we find an almost exact correspondence between the profiles which is in line with the $\delta$-independence of the propagation speed. For $\delta\le\sigma<1$ (see Figure \ref{subfig:a}),  the asymptotic approximation for $I_T$ is not smooth near the leading edge which is expected, and accounted for in the transition region (see Section \ref{sec_4.1}) near $M_T(z)=1-\sigma$, within which the profile is described by (\ref{eqn4.42'}). This discrepancy is however small, decreasing with the value of $\delta$. The inset in Figure \ref{subfig:a} focuses on the behaviour of $I_T(z)$ in the transition region. It shows that the agreement between the numerical solution of the exact problem and  the first order correction, derived from  (\ref{TBP}), is excellent, with the error decreasing in $\delta$. The  solution to (\ref{TBP}) is found by first rescaling $\tilde{H}=(c(\sigma)/b(\sigma)^2)^{1/3}\overline{H},\tilde{m}=(b(\sigma)c(\sigma))^{1/3}m$,
so that the problem becomes independent of $\sigma$. The global constants are then determined via a shooting method to approximately satisfy $c_{\pm\infty}(\sigma){\approx}0.7749(b(\sigma)^2/c(\sigma))^{1/3}$.

For completeness, we also include Figure \ref{fig:LPTWs} which shows the behaviour of the [LPTW]s that describe the propagation of immune cells occurring independently of the mutant epithelial cells. As $\delta$ becomes smaller the solution approaches a step function, in line with Theorem \ref{thm:4.6}; the  propagation speed is as explicitly given in Theorem \ref{thm:4.6}.

\section{Discussion}\label{sec_conclusions}
\begin{figure*}[t]
 \centering

\subfloat[$\sigma<{1}$]{%

\begin{tikzpicture}
    \draw (0, 0) node[inner sep=0] {\includegraphics[width=\linewidth]{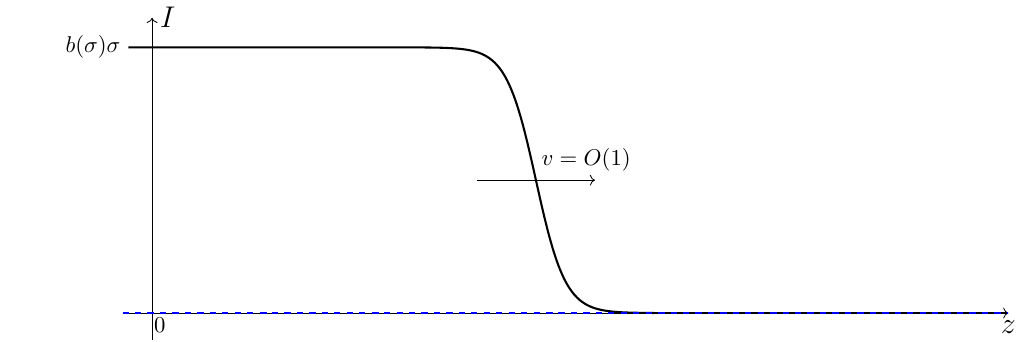}};
    \draw (3, 1) node {Cut-off bistable nonlinearity};
\end{tikzpicture}\label{subfig:Sketcha}%

        }\\
        
        \subfloat[$\sigma\geq{}1$ ]{%
           \begin{tikzpicture}
    \draw (0, 0) node[inner sep=0] {\includegraphics[width=\linewidth]{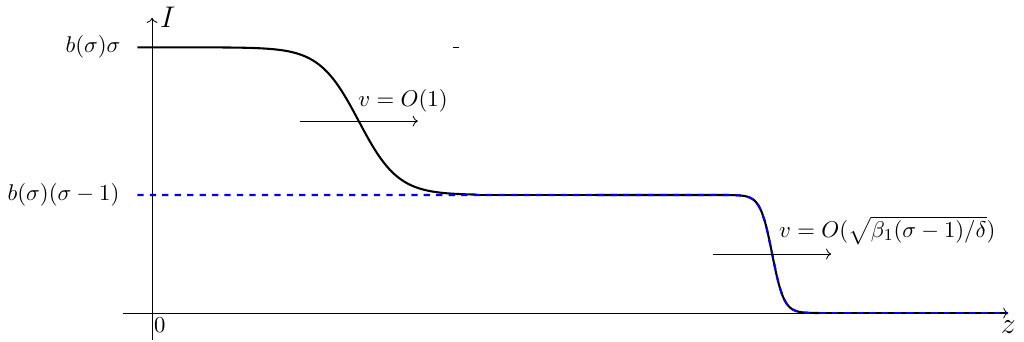}};
    \draw (3, 1) node {Cubic or classical Fisher-type nonlinearity};
\end{tikzpicture}\label{subfig:Sketchb}%

        }
        \caption{The immune component $I$   of the solution to the evolution problem (LIVP) for large-$t$,
         for (a) the anti-inflammatory regime, 
        when the initial mutant density satisfies $M_0>\max(K_M(1-\sigma),0)$ (black solid) and  $M_0\le\max(K_M(1-\sigma),0)$ (blue dashed) and (b) the pro-inflammatory regime, when $M_0>0$ (black solid) and $M_0=0$ (blue dashed). 
      }
        \label{fig:LongTimeSketch}
    \end{figure*}

We have developed a simple  system of reaction-diffusion equations to mechanistically model the spatio-temporal
progression of mutant epithelial cells as they interact with the intestinal barrier, immune cells,
and bacteria. We have focused on answering if, and how, mutant epithelial cells are a significant factor driving the propagation of inflammation in IBD.  Our analysis reveals that they are either essential to, or merely promoting of, inflammation depending on the ratio between growth and decay rates (scaled by constant epithelial cell density), represented by the dimensionless parameter $\sigma$, given by (\ref{eqn:SigmaExpression}).
 This parameter can also be interpreted as the ratio of the strengths of pro-inflammatory processes (immune cell recruitment, damage to the intestinal barrier, bacterial entry) to anti-inflammatory processes (immune cell clearance, barrier repair, and bacterial clearance).

 In Section \ref{sec_IVP} we established that permanent form travelling wave structures may develop as large-$t$ attractors for the initial value problem (LIVP) corresponding to the dynamical system for mutant epithelial cells and immune cells (\ref{eqn:LDS}). In Section \ref{sec_PTW} we   focused on the situation where the ratio of the growth rate mutant epithelial cells compared to the response rate of the adaptive immune system, represented by $\delta$, is assumed small.  We showed that the system supports three distinct families of PTW solutions, namely: full transition PTWs connecting the fully saturated equilibrium $\mathbf{e}_F$ to the unreacted equilibrium  $\mathbf{e}(0)$ ([FPTW]s); upper transition PTWs connecting the fully saturated equilibrium to a transitional equilibrium point $\mathbf{e}_T$ ([UPTW]s); and lower transition PTWs connecting the transitional equilibrium to the unreacted state ([LPTW]s).   In the anti-inflammatory regime, case (i, $\delta<\sigma<1$), there exists a [FPTW]  with unique speed $v=O(1)$, which at leading order is obtained from a scalar reaction-diffusion equation with a cut-off bistable nonlinearity.
  When pro- and anti-inflammatory processes are exactly balanced (ii, $\sigma=1)$ we showed the existence of [FPTW] solutions with minimum speed $v_m=O(1)$ whose reduced dynamics are governed by a cubic Fisher nonlinearity. In the pro-inflammatory regime (iii, $ \sigma>1)$, [FPTW] solutions no longer exist, and we instead showed the existence of [UPTW] with slow minimum speed $v_m^U=O(1)$, and [LPTW] with fast minimum speed $v_m^L=O(1/\sqrt{\delta})$) which are governed by, at leading order and exactly respectively, classical Fisher-type non-linearities. Unlike all other travelling waves we obtained, the [LPTW] solutions are independent of mutant epithelial cells. The ratio of the maximum height of the [UPTW] to the maximum height of the [LPTW] is given by $\sigma/{(\sigma-1)}$. Thus,  the effect of presence of mutant epithelial cells increases  as  $\sigma$ approaches $1$ from above.

We illustrate the $\sigma$-dependent effect of mutant epithelial cells on the expected long-term development of the initial value problem (LIVP) for (\ref{eqn:LDS}) in the anti-inflammatory regime (Fig \ref{subfig:Sketcha}) and the pro-inflammatory regime (Fig \ref{subfig:Sketchb}). In the former case, there is no propagation of inflammation unless the initial mutant density $M_0$ exceeds the threshold value $K_M(1-\sigma)$. Thus, the presence of mutant epithelial cells enables inflammation to propagate in this case, when it otherwise would be unable to do so.  In the pro-inflammatory case, the propagation occurs regardless of the value of $M_0$, meaning that mutant epithelial cells are not essential to the propagation of inflammation. However, they do promote higher levels of inflammation, albeit with subdominant effect on the speed of propagation, in the form of a slow moving moving [UPTW].

 We highlight that case (i, $\delta<\sigma<1$) is the  first instance where the reduction to a cut-off scalar reaction-diffusion equation occurs from  a system of reaction-diffusion equations as a result of separation in timescales for the system (a similar reduction has been hypothesised in \cite{holzer2017wavetrain} where the underlying system is also characterised by a separation in timescale). Note that the value of the dependent variable at the cut-off value is $1-\sigma$, and consequently, our model requires the consideration of an arbitrary cut-off.
 This is to be contrasted with previous works on scalar cut-off reaction-diffusion problems, whose modelling considerations led them to focus on small cut-off values, see for example,  \cite{dumortier2007critical,dumortier2010geometric,popovic2011geometric,popovic2012geometric}.
Our theory covers the full range of available cut-off values.

 We have focused on the permanent form travelling waves solutions to the system (\ref{eqn:LDS}), but it would also be interesting to consider the evolution problem in more detail and determine how the families of PTWs identified in Section \ref{sec_PTW} are established as large-$t$ attractors in the evolution problem (LIVP) for arbitrary $\sigma$. The rate of convergence towards each PTW solution should be studied to determine if PTW solutions emerge on a realistic timescale. Another important direction is to study other reaction functions in (\ref{eqn:LDS}), so as to establish the robustness of our conclusions to modelling choices. A particularly interesting reaction is one where immune cells are recruited proportionally to their interaction rate with bacteria, rather than just the amount of bacteria, leading a reaction rate $\alpha_1BI(1-I/K_I)$ as opposed to $\alpha_1B(1-I/K_I)$. We expect to obtain a discontinuous bistable cut-off reaction function in the reduced dynamics for some parameter regimes, requiring some modifications to the asymptotic analysis developed herein.

  By highlighting how pro- and anti-inflammatory processes interact with mutant epithelial cells to drive inflammation, our work has potential for accelerating the development of therapies which target either pro-inflammatory molecules (i.e as TNF$\alpha$), or  anti- inflammatory molecules, such as pro-solving lipid mediators (see \cite{schoultz2019cellular} for a more complete list of therapeutic targets). A key prediction, which is a consequence of the $\sigma$-dependent effect of mutant epithelial cells discussed above, is that therapies can halt the spread of inflammation by bringing the system into the anti-inflammatory regime. This can be achieved  by targeting the processes associated with the parameters in the closed form expression  (\ref{eqn:SigmaExpression}) for $\sigma$. Currently, treatments frequently use anti-TNF$\alpha$  antibodies to reduce immune cell activity and thus the damage done to the intestinal barrier \cite{schoultz2019cellular} (decreasing $\alpha_2$). Such treatments can also indirectly increase apoptosis of immune cells \cite{levin2016mechanism} (increasing $\beta_1$). Other state-of-the-art treatments, like anti-adhesion agents, target the ability of immune cells to infiltrate the intestinal barrier \cite{chudy2019state} (decreasing $\alpha_1$). However, it may be necessary to decrease $\sigma$ even further below $1$ since the necessary condition for propagation $M_0\geq{}\max(K_M(1-\sigma),0)$ may still be satisfied even for small $\sigma$.

Of course, one may want to develop a more detailed model of inflammation. However, we anticipate that the slow growth of the mutant epithelial cells means that the structural theoretical approach developed in this paper may well be adapted to readily analyse such models in a similar manner, allowing the determination of the boundary between health and disease.
\bmhead{Acknowledgements}

F.S. acknowledges funding for a UKRI Future Leaders Fellowship (MR/T043571/1). BvR thanks the University of Birmingham for support.



\section*{Declarations}
\textbf{Conflict of interest:} The authors have no relevant financial or non-financial interests to disclose.

\bibliography{bibliography}
\begin{appendix}\label{appendix}
    \section{Proofs from Section \ref{sec_IVP}}\label{app_A}
    In this appendix we provide those proofs defered from subsections 3.1 and 3.2. First we require a preliminary result. We set $M_a(M_0,\sigma)=M_0 + (1/2)((1-\sigma)-M_0)\in (M_0,(1-\sigma))$ and introduce, under the parameter restrictions of subsection 3.1,\\

\begin{lem} \label{lem}
With $\delta$ small, let $(M,I):\overline{D}_{\infty}\to R(\delta)$ be the solution to (LIVP). Then
\begin{equation}
    M(x,t) < M_a(M_0,\sigma)~~\forall~~(x,t)\in \overline{D}_{\infty}. \label{eqn3.14}
\end{equation}
\end{lem}
\begin{proof}
    Suppose the contrary holds, then since $M_a(M_0,\sigma)>M_0 + (1/2)\delta^{\frac{1}{4}}$ and $M_{-\infty}(t)\le M_0 + O(\delta)~~\forall~~t\ge 0$, it follows that there is a point $(x,t)=(x_a,t_a)$ such that,
    \begin{equation}
        0<M(x,t)<M_a(M_0,\sigma)~~\forall~~(x,t)\in \mathbb{R} \times (0,t_a),  \label{eqn3.15}
    \end{equation}
    whilst,
    \begin{equation}
        M(x_a,t_a) = M_a(M_0,\sigma). \label{eqn3.16}
    \end{equation}
    It then follows from (\ref{eqn1.17}) that,
    \begin{equation}
     f(M(x,t),I(x,t)) \le -c(M_0,\sigma)I(x,t)~~\forall~~(x,t)\in \overline{D}_{t_a}. \label{eqn3.17}   
    \end{equation}
    where,
    \begin{equation}
        c(M_0,\sigma) = \frac{a(\sigma)b(\sigma)((1-\sigma)-M_0)}{2(\alpha_2 + \beta_2)}>0. \label{eqn3.18} 
    \end{equation}
    A straightforward application of the parabolic comparison theorem on $\overline{D}_{t_a}$, with the parabolic operator $\mathcal{N}[w]\equiv w_t - Dw_{xx} + \delta^{-1}c(M_0,\sigma)w$, then establishes, via (\ref{eqn3.17}), that,
    \begin{equation}
        0<I(x,t) < I_0 \exp{(-\delta^{-1}c(M_0,\sigma)t)}~~\forall~~(x,t)\in D_{t_a}. \label{eqn3.19}   
    \end{equation}
    We now make another straightforward application of the parabolic comparison theorem on $\overline{D}_{t_a}$, this time with the parabolic operator $\mathcal{N}[w]\equiv w_t - w_{xx} -  I_0 \exp{(-\delta^{-1}c(M_0,\sigma)t)}w$, to establish that, via  (\ref{eqn3.19}), (\ref{eqn1.15}) and (\ref{eqn3.15}),
    \begin{equation}
        M(x,t) < M_0\exp{\left(I_0\frac{\delta}{c(M_0,\sigma)}\left(1 - \exp{(-\delta^{-1}c(M_0,\sigma)t)}\right)\right)}~~\forall~~(x,t)\in D_{t_a}, \label{eqn3.20} 
    \end{equation}
    which gives, with $\delta$ small (recalling that $((1-\sigma) -M_0) > \delta^{\frac{1}{4}}$),
    \begin{equation}
        M(x,t)< M_0\left(1 + \frac{4}{c(M_0,\sigma)}\delta \right)~~\forall~~(x,t)\in D_{t_a}, \label{eqn3.21}
    \end{equation}
    and, in particular,
    \begin{equation}
        M(x_a,t_a) < M_0\left(1 + \frac{4}{c(M_0,\sigma)}\delta \right). 
    \end{equation}
    Thus, from (\ref{eqn3.16}),
    \begin{equation}
        M_a(M_0,\sigma) -  M_0\left(1 + \frac{4}{c(M_0,\sigma)}\delta \right) < 0. \label{eqm3.22} 
    \end{equation}
    However, with $\delta$ suitably small, 
    \begin{multline}
      M_a(M_0,\sigma) -  M_0\left(1 + \frac{4}{c(M_0,\sigma)}\delta \right) = (1/2)((1-\sigma)-M_0) \left(1 - \frac{8M_0 \delta}{c(M_0,\sigma) ((1-\sigma)-M_0)}\right)\\
      \ge (1/2)\delta^{\frac{1}{4}} \left(1 - \frac{16\delta^{\frac{1}{2}}(\alpha_2+\beta_2)}{\beta_1 \beta_2}\right) > 0
      \end{multline}
    and we arrive at a contradiction, which completes the proof.    
\end{proof}
We now give the respective proofs.

\subsection{Proof of Theorem \ref{thm3.2} }

    The left side of the first inequality follows directly from Theorem 3.1, whilst the left side of the second inequality follows from Theorem 3.1 and Lemma \ref{lem}, with an application of the parabolic comparison theorem with the operator $\mathcal{N}[w]\equiv w_t - w_{xx}$. The right side inequality in (\ref{eqn3.23}) follows again from Theorem 3.1, Lemma \ref{lem}  and an applicaton of the parabolic comparison theorem with the operator $\mathcal{N}[w]\equiv w_t - Dw_{xx} + \delta^{-1}c(M_0,\sigma)w$. Correspondingly, the right side inequality in (\ref{eqn3.24}) follows from Theorem 3.1, Lemma \ref{lem}  and an applicaton of the parabolic comparison theorem now with the operator $\mathcal{N}[w]\equiv w_t - w_{xx} + I_0 \exp{(-\delta^{-1}c(M_0,\sigma)t)} w$. The proof is complete.

\subsection{Proof of Theorem \ref{thm3.3}}
The first inequalities follow from Theorem 3.1, and applications of the parabolic comparison theorem, with operators $\mathcal{N}[w]\equiv w_t - w_{xx}$ and $\mathcal{N}[w]\equiv w_t - w_{xx} - w(1-w)$ respectively. Correspondingly the second inequalites are obtained by using Theorem 3.1 and the operator $\mathcal{N}[w]\equiv w_t - w_{xx} - a(\sigma)\alpha_2^{-1}(b(\sigma)\sigma - w)$ in the parabolic comparison theorem. The proof is complete.

\section{Proofs from Section \ref{sec_PTW}}\label{app_B}
Here we provide proofs to Propositions 4.1-4.4 and Theorem 4.9

\subsection{Proof of Proposition 4.1}
Suppose that there exists $z_0\in \mathbb{R}$ such that $M_T(z_0)=0$, then since $M_T(z)\in [0,1]$, we must have $M_T'(z_0)=0$. However, this forms a regular initial value problem for $M_T(z)$ at $z=z_0$ with equation (\ref{eqn4.3}), which has the trivial solution as well as $M_T(z)$. Uniqueness for this initial value problem then requires that $M_T(z)=0$ for all $z\in \mathbb{R}$, which contradicts condition (\ref{eqn4.5}). The result for $I_T(z)$ follows similarly.
\subsection{Proof of Proposition 4.2}
The proof for $M_T(z)$ follows exactly the strategy in the last proof, and is left to the reader. For $I_T(z)$, suppose false, then there is a point $z^*\in \mathbb{R}$ such that
\begin{equation}
    I_T(z^*)\ge b(\sigma)\sigma~~\text{and}~~I_T'(z^*)=0, \label{eqn4.8}
\end{equation}
with 
\begin{equation}
    I_T''(z^*)\le 0,  \label{eqn4.9}
\end{equation}
whilst, from the first part of this result, proved above, $M_T(z^*)<1$. However, from  equation (\ref{eqn4.4}) with (\ref{eqn4.8}), we have,
\begin{equation}
    I_T''(z^*) = -\frac{(b(\sigma)(M_T(z^*)-1) + (b(\sigma)\sigma-I_T(z^*))}{\overline{\delta}(\alpha_2I_T(z^*) + \beta_2(1-M_T(z^*))}>0
\end{equation}
via (\ref{eqn4.8}) and Proposition 4.1, which contradicts (\ref{eqn4.9}). The result follows.
\subsection{Proof of Proposition 4.3}
    For $M_T(z)$, suppose false. Then there exists $z^*\in \mathbb{R}$ such that
    \begin{equation}
        M_T'(z^*)=0~~\text{and}~~M_T''(z^*)\ge0.   \label{eqn4.11}
    \end{equation}
    However, equation (\ref{eqn4.3}) with (\ref{eqn4.11})$_1$, gives,
    \begin{equation}
        M_T''(z^*) = -I_T(z^*)M_T(z^*)(1-M_T(z^*))<0
    \end{equation}
    via Proposition 4.1 and Proposition 4.2, which contradicts (\ref{eqn4.11})$_2$, and the result follows. For $I_T(z)$, again suppose false. Then there exists $z_1^*\in \mathbb{R}$ such that
    \begin{equation}
        I_T'(z_1^*)=0~~\text{and}~~I_T''(z_1^*)\ge0,   \label{eqn4.13}
    \end{equation}
    followed by $z_2^* > z_1^*$ with,
    \begin{equation}
        0<I_T(z_1^*) \le I_T(z_2^*)<b(\sigma)\sigma,  \label{eqn4.14}
    \end{equation}
    and
    \begin{equation}
        I_T'(z_2^*)=0~~\text{and}~~I_T''(z_2^*)\le 0. \label{eqn4.15}
    \end{equation}
   where use has also been made of Propositions 4.1 and 4.2. From the first part of this result, proved above, and Propositions 4.1 and 4.2, we also have,
    \begin{equation}
      0 < M_T(z_2^*) < M_T(z_1^*) < 1.  \label{eqn4.16}  
    \end{equation}
   Now, from equation (\ref{eqn4.4}) we have, using (\ref{eqn4.13})$_1$,
   \begin{equation}
       I_T''(z_1^*) = -\frac{1}{\overline{\delta}}f(M_T(z_1^*),I_T(z_1^*)), \label{eqn4.17}
   \end{equation}
   and so, from (\ref{eqn4.13})$_2$, we must have,
   \begin{equation}
    f(M_T(z_1^*),I_T(z_1^*)) \le 0.  \label{eqn4.18}   
   \end{equation}
   We now observe that on $(X,Y)\in [0,1]^2$, we have the disjoint subsets,
   \begin{multline}
       A = \{(X,Y)\in [0,1]^2: f(X,Y) < 0\}\\ = \{(X,Y)\in [0,1]^2: 0\le X\le 1,~ Y > \text{max}(0, (b(\sigma)(\sigma-1) +b(\sigma)X) \},
   \end{multline}
   \begin{multline}
   B = \{(X,Y)\in [0,1]^2: f(X,Y) = 0\}\\ = \{(X,Y)\in [0,1]^2: 0\le X\le 1,~ Y = \text{max}(0, (b(\sigma)(\sigma-1) +b(\sigma)X) \},
   \end{multline}
   \begin{multline}
   C = \{(X,Y)\in [0,1]^2: f(X,Y) > 0\}\\ = \{(X,Y)\in [0,1]^2: (1-\sigma)\le X\le 1,~ 0<Y<(b(\sigma)(\sigma-1) +b(\sigma)X) \}.
   \end{multline}
   Therefore, we must have, from (\ref{eqn4.18}),
   \begin{equation}
     (M_T(z_1^*),I_T(z_1^*)) \in A\cup C,  
   \end{equation}
   and so at $z=z_2^*$ then,
   \begin{equation}
       (M_T(z_2^*),I_T(z_2^*)\in A,
   \end{equation}
   via the first part, proved above. Thus,
   \begin{equation}
       f(M_T(z_2^*),I_T(z_2^*)) < 0, \label{eqn4.23}
   \end{equation}
   and so via equation (\ref{eqn4.3}), with (\ref{eqn4.15})$_1$, we have,
   \begin{equation}
       I_T''(z_2^*) = - \frac{f(I_T(z_2^*),M_T(z_2^*))}{\overline{\delta}} > 0,
   \end{equation}
   which contradicts (\ref{eqn4.15})$_2$. The result follows.
   \subsection{Proof of Proposition 4.4}
    For $I_T(z)$ the proof is exactly the same as in Proposition 4.1. Now consider $M_T:\mathbb{R}\to [0,1]$. From the first equation in \eqref{eqn:EVP2} we have,
    \begin{equation}
      (M_T'(z)\exp{(vz)})' = (M_T(z)-1)M_T(z)I_T(z)\exp{(vz)} \le 0~~\forall~~z\in \mathbb{R},  
    \end{equation}
    via the first statement above and conditions in \eqref{eqn:EVP2}. Thus, on integration, we have,
    \begin{equation}
        M_T'(z)\le 0~~\forall~~z\in \mathbb{R},
    \end{equation}
    which gives, using the condition as $z\to -\infty$,
    \begin{equation}
        M_T(z)\le 0~~\forall~~z\in \mathbb{R},
    \end{equation}
    and so,
    \begin{equation}
        M_T(z)=0~~\forall~~z\in \mathbb{R},
    \end{equation}
    as required.

   \subsection{Proof of Theorem 4.9}
    We return to \eqref{eqn:EVP1}. Suppose that a [FPTW] solution exist, say $(M_T,I_T):\mathbb{R}\to [0,1]^2$. Then it is straightforward to establish that Propositions 4.1-4.3 hold. In addition, as $z\to \infty$, linearization requires that,
    \begin{equation}
        \delta(DI_T'' + vI_T')+ \beta_1(\sigma-1)I_T \sim 0,~~z\gg1,
    \end{equation}
    \begin{equation}
        I_T(z)\to 0~~\text{as}~~z\to \infty.
    \end{equation}
    Since $I_T$ is positive, this establishes that the propagation speed $v$ must satisfy the inequality,
    \begin{equation}
        v\ge 2\sqrt{D(\sigma-1)\beta_1/\delta}. \label{eqn4.129}
    \end{equation}
    Moreover, via Propositions 4.1-43, we infer that there exists a translation of the [FPTW] so that,
    \begin{equation}
        I_T(z) \le \exp{}\left(-\frac{1}{2D}(v - \sqrt{v^2 - 4D(\sigma-1)\beta_1\delta^{-1}})z\right)
        ~~\forall~~z\ge 0
    \end{equation}
    which can be reduced further under (\ref{eqn4.129}) to,
    \begin{equation}
        I_T(z)\le \exp{}(-(\sigma-1)\beta_1 v^{-1}\delta^{-1}z)~~\forall~~z\ge 0. \label{eqn4.131}
    \end{equation}
    We conclude from (\ref{eqn4.129}) that with $\delta$ small, the propagation speed $v$ must be large. Correspondingly, when $v$ is large, we have that, 
    \begin{equation}
        M_{T\xi} \sim- I_T(\xi)M_T(1-M_T), \label{eqn4.132}
    \end{equation}
    uniformly for $\xi\in \mathbb{R}$, with $\xi= z/v$. Also, via (\ref{eqn4.131}), we have,
    \begin{equation}
        I_T(\xi)\le \exp{}(-(\sigma-1)\beta_1\delta^{-1}\xi)~~\forall~~\xi\ge 0, \label{eqn4.133}
    \end{equation}
    and so,
    \begin{equation}
        \int_{0}^{\xi}{I_T(\xi')}d\xi' \le \frac{\delta}{(\sigma-1)\beta_1}\left(1 - \exp{}(-(\sigma-1)\beta_1\delta^{-1}\xi)\right)~~\forall~~\xi\ge 0. \label{eqn4.134} 
    \end{equation}
    Therefore, $I_T(\xi)$ is improper integrable, and,
    \begin{equation}
        I_{\infty}\equiv \int_{0}^{\infty}{I_T(\xi')}d\xi' \le \frac{\delta}{(\sigma-1)\beta_1}.  \label{eqn4.135}
    \end{equation}An integration of equation (\ref{eqn4.132}) now leads to,
    \begin{equation}
    M_T(\xi) \sim\frac{M_0 ~\exp{}\left(-\int_{0}^{\xi}{I_T(\xi')}d\xi'\right)}{\left((1-M_0) + M_0~\exp{}\left(-\int_{0}^{\xi}{I_T(\xi')}d\xi'\right)\right)}~~\forall~~\xi\in \mathbb{R}, \label{eqn4.136}
    \end{equation}
    with $M_0 = M_T(0)\in (0,1)$. However, we now obtain from (\ref{eqn4.135}) and (\ref{eqn4.136}), that (with $\delta\le \text{min}(1/2,(\sigma-1)\beta_1/2)$),
    \begin{equation}
        M_T(\xi) \to \frac{M_0~\exp{}(-I_{\infty})}{(M_0 + (1-M_0)~\exp{}(-I_{\infty}))}\ge M_0\left(1 - \frac{\delta}{(\sigma-1)\beta_1)}\right)\ge \frac{1}{2}M_0>0 ~~\text{as}~~\xi\to \infty.
    \end{equation}
    However, since $M_T(\xi)$ is a [FPTW], then  $M_T(\xi)\to 0$ as $\xi\to \infty$, and we obtain a contradiction. The proof is complete.

 \section{Derivation of \texorpdfstring{$v^*(\sigma)$}{} as \texorpdfstring{$\sigma\to 0^+$}{} and \texorpdfstring{$\sigma\to1^-$}{}}\label{app_C}
We here consider the problem \eqref{eqn:MVP1}, which we recall as the following boundary value problem, 
\begin{equation}
    \tilde{M}_0'' + v\tilde{M}_0' + F(\tilde{M}_0) = 0,~~z\in \mathbb{R},
\end{equation}
subject to,
\begin{equation}
    \tilde{M}_0(z)\in (0,1)~~\forall~~z\in \mathbb{R},
\end{equation}
\begin{equation}
    \tilde{M}_0(z)\to
    \begin{cases}
        1~~\text{as}~~z\to -\infty,\\
        0~~\text{as}~~z\to \infty,
    \end{cases}
\end{equation}
\begin{equation}
    \tilde{M}_0(0) = (1-\sigma).
\end{equation}

The function $F(X)$ is defined as $F(X)=H_{0}(X)X(1-X)$ where
\begin{equation}
    H_0(X) =
    \begin{cases}
        b(\sigma)(X - (1-\sigma)),~X\in ((1-\sigma),1]\\
        0,~X\in [0,(1-\sigma)].
    \end{cases}
    \label{H0}
\end{equation}

We may simplify this problem by introducing $\tilde{z}=1/\sqrt{b(\sigma)}$, and $\tilde{v}=v\sqrt{b(\sigma)}$. The problem then becomes
\begin{subequations}\label{eqn:RescapeMVp1}

\begin{equation}
    \tilde{M}_0'' + \tilde{v}\tilde{M}_0' + \tilde{F}(\tilde{M}_0) = 0,~~z\in \mathbb{R},
\end{equation}
subject to,
\begin{equation}
    \tilde{M}_0(z)\in (0,1)~~\forall~~z\in \mathbb{R},
\end{equation}
\begin{equation}
    \tilde{M}_0(z)\to
    \begin{cases}
        1~~\text{as}~~z\to -\infty,\\
        0~~\text{as}~~z\to \infty,
    \end{cases}
\end{equation}
\begin{equation}
    \tilde{M}_0(0) = (1-\sigma).
\end{equation}

The function $\tilde{F}(X)$ is defined as $\tilde{F}(X)=\tilde{H}(X)X(1-X)$ where
\begin{equation}
    \tilde{H}(X) =
    \begin{cases}
        (X - (1-\sigma)),~X\in ((1-\sigma),1]\\
        0,~X\in [0,(1-\sigma)].
    \end{cases}
    \label{Htilde}
\end{equation}
\end{subequations}

Our principal objective is to determine asymptotic approximations to the propagation speed $v=v^{*}(\sigma)$ ($\tilde{v}=\tilde{v}^{*}(\sigma)$)
in the  two complementary asymptotic limits $\sigma\rightarrow{0}^{+}$ and $\sigma\rightarrow{1}^{-}$. 
We first examine (\ref{eqn:RescapeMVp1}) in the limit $\sigma\rightarrow{0}^{+}$. This is best achieved in the $(\tilde{M}_0,\tilde{M}'_0)$ phase plane. Thus we set $\alpha=\tilde{M}_0$ and $\beta=\tilde{M}'_0$, and now follow closely the corresponding theory laid out in \cite{TNT}.
The theory in \cite{TNT} guarantees that $\tilde{v}^{*}(\sigma)=o(1)$ as $\sigma\rightarrow 0^{+}$, and 
moreover, that solutions to (\ref{eqn:RescapeMVp1}) are equivalent to phase paths $\beta=\beta(\alpha;\sigma)$ which satisfy the boundary value problem,
\begin{subequations}
       \begin{align} 
		   \frac{d\beta}{d\alpha}&=-\tilde{v}^{*}(\sigma)-\frac{\tilde{F}(\alpha;\sigma)}{\beta},~\alpha\in (1-\sigma,1),\label{eqnc3} \\
      \beta(\alpha;\sigma)&\sim-\lambda_{+}(\tilde{v}^{*}(\sigma))(1-\alpha) 
 ~~\text{as}~~\alpha\rightarrow{1}^{-}, \label{eqnc1}\\ 
        \beta(1-\sigma;\sigma)&=-\tilde{v}^{*}(\sigma)(1-\sigma), \label{eqnc2}
		\end{align}
\end{subequations}
where 
\begin{equation}
  \lambda_{+}(\tilde{v}^{*}(\sigma)) = \frac{1}{2}(-\tilde{v}^*(\sigma) + \sqrt{\tilde{v}^*(\sigma)^2 + |\tilde{F}'(1;\sigma)|}).  
\end{equation}
The first boundary condition  follows from the local (linear) structure of the unstable manifold of the saddle equilibrium point at $\mathbf{e}_1 =(1,0)$
with $\lambda_{+}(\tilde{v}^{*}(\sigma))$ corresponding to the positive eigenvalue.
The second  boundary condition  follows from the requirement to connect $\mathbf{e}_1$ with the single (one-dimensional) stable manifold associated with the non-hyperbolic equilibrium point $\mathbf{e}(0)=(0,0)$.
Next, in the limit $\sigma\to 0^+$ we have $|\tilde{F}'(1;\sigma)|=\sigma=O(\sigma)$. Thus, in this limit $\lambda_{+}(\tilde{v}^{*}(\sigma))=o(1)$ which, when combined with two the boundary conditions, implies that $\beta=o(1)$
as $\sigma\to 0^+$. 
There are two possible asymptotic behaviours for $\lambda_{+}(\tilde{v}^{*}(\sigma))$ as 
$\sigma\to 0^+$, dependent on the magnitude of $\tilde{v}^{*}(\sigma)$ relative to 
$|\tilde{F}'(1;\sigma)|$. However, after considering both boundary conditions, we deduce that 
$\lambda_{+}(\tilde{v}^{*}(\sigma))=O(\sigma^{\frac{1}{2}})$ as $\sigma\to 0^+$ is the only possible behaviour, resulting in a nontrivial asymptotic balance in the ODE for $\beta$.
In this case, using \eqref{eqnc1}, $\beta=O(\sigma^{3/2})$ as $\sigma\to 0^+$
while from  \eqref{eqnc2}, $\beta=O(\tilde{v}^*(\sigma))$. We thus conclude that
$\tilde{v}^*(\sigma)=O(\sigma^{3/2})$ as $\sigma\to 0^+$. This leads to the following rescalings,
\begin{equation}
    \beta=\sigma^{\frac{3}{2}}{Y},\ \
    \alpha=1-\sigma{X},\ \
    \tilde{v}^{*}(\sigma)=\sigma^{\frac{3}{2}}V(\sigma).
\end{equation}	
The boundary value problem \eqref{eqnc3}-\eqref{eqnc2} then becomes
\begin{subequations}
    \begin{align}
        \label{eqn:PhasePathSystemRescale}
    \frac{dY}{dX}&=\sigma V(\sigma)+\frac{(1-\sigma X)X(1-X)}{Y},
	~~X\in (0,1), 
	\\
        Y(X;\sigma)&\sim-\sigma^{-\frac{1}{2}}\lambda_{+}(\sigma^{\frac{3}{2}}V(\sigma))X~~\text{as}~~ X\rightarrow{0}^{+},\\
         Y(1;\sigma)&=-V(\sigma)(1-\sigma).
		 \end{align}
\end{subequations} 
We now  expand $Y(X;\sigma)$ and $V(\sigma)$ 
according to $Y=Y_{0}(X)+O(\sigma^{\frac{1}{2}})$, and $V(\sigma)=V_{0}+O(\sigma^{\frac{1}{2}})$ as $\sigma\to 0^+$, with $Y_0,V_0=O(1)$. At leading order we obtain the following boundary value problem for 
$Y_0(X)$, namely, 
\begin{subequations}\label{eqn:FOP}
\begin{align}
    \frac{dY_{0}}{dX}&=\frac{(1-X)X}{Y_{0}},~~X\in(0,1),\label{eqn:FO1} \\
    Y_{0}(X)&\sim- X~~\text{as}~~X\rightarrow{0}^{+}, \label{eqn:FO2} \\
    Y_{0}(1)&=-V_{0}.\label{eqn:FO3}
	\end{align}
\end{subequations}
The general solution to \eqref{eqn:FOP} is $Y_{0}(X)^2=c_{1}+ X^2\left(1-2X/3\right)$ where $c_1$ is an arbitrary constant of integration. Upon application of the boundary condition \eqref{eqn:FO2} we obtain that $c_1=0$. Then, 
\begin{equation}
	Y_{0}(X)=-X\left((1-2 X/3)\right)^{\frac{1}{2}},~~X\in[0,1].
\end{equation}	
We now use the boundary condition \eqref{eqn:FO3} to deduce that 
\begin{equation}
    V_{0}=\sqrt{1/3}. 
\end{equation}
Therefore, at leading order, the speed of propagation is given by
\begin{equation}\label{eqn:AsymptoticSpeedSmallSigma} 
	v^{*}(\sigma)=(\sqrt{b(\sigma)/3})\sigma^{3/2}+O(b(\sigma)^{1/2}\sigma^{2})~~\text{as}~~\sigma\rightarrow{0}^{+}.
\end{equation}

We next examine (\ref{eqn:RescapeMVp1}) in the limit $\sigma\rightarrow{}1^{-}$. We let $\eta=1-\sigma$ so that the problem becomes
\begin{subequations}
    \begin{equation}
    \label{eqn:MEqn2}
    \tilde{M}''+\tilde{v}^{*}(\eta)\tilde{M}'+\tilde{F}(\tilde{M};1-\eta)=0,
\end{equation}
\begin{equation}\label{eqn:MEqnBC}
     \tilde{M}(\tilde{z})\rightarrow\begin{cases}
        1~~\text{as}~~ \tilde{z}\rightarrow\infty,\\ 
                0~~\text{as}~~\tilde{z}\rightarrow\infty.
    \end{cases}
\end{equation}
\end{subequations}
We again consider the problem in the phase plane. Similarly to before, we obtain:
\begin{subequations}\label{eqn:SCOSystem1}
       \begin{align} 
		   \frac{d\beta}{d\alpha}&=-\tilde{v}^{*}(\eta)-\frac{\alpha(\alpha-\eta)(1-\alpha)}{\beta}, \label{eqn:SCOPhasePathSystem1} \\
      \beta(1;\eta)&=0,  \label{eqn:SCOPhasePathSystem2} \\ 
        \beta(\eta;\eta)&=-\tilde{v}^{*}(\eta)\eta\label{eqn:SCOPhasePathSystem3}.
		\end{align}
\end{subequations}
We now examine this boundary value problem to determine the asymptotic structure of $v^{*}(\eta)$ ($\tilde{v}^{*}(\eta)$) as $\eta\rightarrow{0}^{+}$. It follows from Theorem 4.2 that the speed $\tilde{v}^{*}(\eta)$ of the PTW solution  may be written as
\begin{equation}
    \tilde{v}^{*}(\eta)=\frac{1}{\sqrt{2}}-\tilde{v}(\eta),
\end{equation}
where
$\tilde{v}(\eta)=o(1)\text{ as }\eta\rightarrow{0}.$

To find the first order approximation we consider the $\eta=0$ problem, 
\begin{subequations}\label{eqn:0SCOSystem1}
       \begin{align} 
		   \frac{d\beta}{d\alpha}&=-\frac{1}{\sqrt{2}}-\frac{\alpha^2(1-\alpha)}{\beta}, \label{eqn:0SCOPhasePathSystem1_0} \\
      \beta(1;0)&=0,  \label{eqn:0SCOPhasePathSystem2_0} \\ 
        \beta(0;0)&=0\label{eqn:0SCOPhasePathSystem3_0},
		\end{align}
\end{subequations}
for which the solution is $\beta_{0}=-\frac{1}{\sqrt{2}}\alpha(1-\alpha)$. This solution satisfies \eqref{eqn:SCOPhasePathSystem2} up to order $\eta$ as $\eta\rightarrow{0}^{+}$. It then follows from \eqref{eqn:SCOPhasePathSystem3} that $\tilde{v}(\eta)$ can be written as
\begin{equation}\label{eqn:tildev}
    \tilde{v}(\eta)=\eta{}\tilde{v}_{1}+o(\eta)\text{ as }\eta\rightarrow{0}^{+},
\end{equation}
for a constant $\tilde{v}_1$.
We can similarly expand $\beta$ as:
\begin{equation}\label{eqn:beta_expansion}
    \beta=\beta_0+\eta\beta_1+o(\eta)\text{ as }\eta\rightarrow{0}^{+}.
\end{equation}
Substituting \eqref{eqn:tildev} and \eqref{eqn:beta_expansion} into \eqref{eqn:0SCOSystem1} we obtain, as $\eta\rightarrow{0}^{+},$
\begin{align*}
    \frac{d\beta_0}{d\alpha}+\eta\frac{d\beta_1}{d\alpha}&=-\frac{1}{\sqrt{2}}+\eta{}\tilde{v}_1-\frac{\alpha^2(1-\alpha)}{\beta_0+\eta\beta_1}+\eta\frac{\alpha(1-\alpha)}{\beta_0+\eta\beta_1}+o(\eta)\\
    &=-\frac{1}{\sqrt{2}}+\eta{}\tilde{v}_1-\frac{\alpha^2(1-\alpha)}{\beta_0}+\eta\frac{\alpha^2(1-\alpha)\beta_1}{\beta_0^2}+\eta\frac{\alpha(1-\alpha)}{\beta_0}+o(\eta).
\end{align*}
The equation for $\beta_1$ is therefore
\begin{subequations}\label{eqn:beta_1_eqn}
       \begin{align} 
		   \frac{d\beta_1}{d\alpha}&=\tilde{v}_1-\sqrt{2}+\frac{2\beta_{1}}{(1-\alpha)}, \label{eqn:0SCOPhasePathSystem1_1} \\
      \beta_{1}(1)&=0,  \label{eqn:BCBeta1_1} \\ 
    \eta\beta_1(\eta)&=O(\eta^2)\label{eqn:BCBeta2_2}\text{ as }\eta\rightarrow{0}^{+}.
		\end{align}
\end{subequations}The solution of \eqref{eqn:0SCOPhasePathSystem1_1} which satisfies {\eqref{eqn:BCBeta1_1}} is $\beta_1=\frac{1}{3}(\sqrt{2}-v_1)(1-\alpha)$. 

We anticipate a boundary layer {over which condition (\ref{eqn:BCBeta2_2}) is satisfied} at $\alpha=\eta$, and notice that \eqref{eqn:BCBeta2_2} implies $\beta=O(\eta)$ and $\alpha=O(\eta)$ at this boundary as $\eta\rightarrow{0}^{+}$. We therefore introduce the scaled variables
\[\tilde{\beta}=\eta^{-1}\beta,A=\eta^{-1}\alpha,\]
where we may expand $\tilde{\beta}=\tilde{\beta}_0+o(\eta)$. We obtain the problem in terms of inner variables:
\begin{subequations}\label{eqn:beta_1_eqn_in}
       \begin{align} 
	\frac{d\tilde{\beta}}{dA}&=-\tilde{v}^{*}(\eta)-\eta\frac{A(A-1)(1-\eta{A})}{\tilde{\beta}}, \label{eqn:0SCOPhasePathSystem1} \\
      \tilde{\beta}(1;\eta)&=-\tilde{v}^{*}(\eta),  \label{eqn:BCBeta1_1_in} 
      \end{align}
\end{subequations}
The first order problem is is trivial and solved by {$\tilde{\beta}_0=-\frac{1}{\sqrt{2}}A$}. An application of Van Dyke's matching principle leads us to conclude $\tilde{v}_1=\sqrt{2}$. Consequently we may express the speed $v^{*}(\sigma)$ in terms of $\sigma$ as
\begin{equation}\label{eqn:AsymptoticSpeedLargeSigma}
    {v}^{*}(\sigma)=b(\sigma)^\frac{1}{2}\left(\frac{1}{\sqrt{2}}-\sqrt{2}(1-\sigma)\right)+O(b(\sigma)^{1/2}(1-\sigma)^2)\text{ as }\sigma\rightarrow{1}^{-}.
\end{equation}

     \end{appendix}

\newpage

\end{document}